\documentclass[a4paper]{article}

\usepackage[pages=all, color=black, position={current page.south}, placement=bottom, scale=1, opacity=1, vshift=5mm]{background}\SetBgContents{}

\usepackage[margin=1in]{geometry} 

\RequirePackage{fix-cm}

\usepackage{amsmath,amsfonts,amssymb,epsfig}
\usepackage{graphicx}
\usepackage{amssymb}
\usepackage{mathptmx}      
\usepackage{mathtools} 
\usepackage{algorithm}
\usepackage[noend]{algpseudocode}

\usepackage[T1]{fontenc}
\usepackage[utf8]{inputenc}
\usepackage[english]{babel}
\usepackage{setspace}
\usepackage{amsthm}
\usepackage{array}
\usepackage{float}
\usepackage{graphicx}
\usepackage{bm}
\usepackage{fancyhdr}
\usepackage{longtable}
\usepackage{epstopdf}
\usepackage{hyperref}
\usepackage{tikz}
\usepackage{pgfplots}
\usepackage{color}
\usepackage{float}
\usepackage{multicol}
\usepackage{latexsym}
\usepackage{hyperref}

\newtheorem{Remark}{Remark}

\newtheorem{thm}{Theorem}
\newtheorem{corollary}[thm]{Corollary}

\newtheorem{lemma}[thm]{Lemma}
\newtheorem{prop}[thm]{Proposition}

\newtheorem{assumption}{Assumption}
\newtheorem{defn}{Definition}
\newtheorem{exmp}{Example}

\def\ItemNN$#1${\item $\displaystyle#1$}

\newcommand{\itemEq}[1]{%
         \begingroup%
         \setlength{\abovedisplayskip}{0pt}%
         \setlength{\belowdisplayskip}{0pt}%
         \parbox[c]{\linewidth}{\begin{flalign}#1&&\end{flalign}}%
         \endgroup}

\DeclareMathOperator*{\argmin}{argmin}

\makeatletter
\def\BState{\State\hskip-\ALG@thistlm}
\makeatother

\usepackage{enumitem} 
\usepackage{marginnote}

\title{Projection in negative norms and the regularization of rough linear functionals}

\author{ F.~Millar $^a$       \and
        I.~Muga $^a$\and 
        S.~Rojas$^b$  \and
        K.G.~van der Zee$^c$
}

\date{\noindent $^a$ Instituto de Matem\'aticas,  
              Pontificia Universidad Cat\'olica de Valpara\'iso, Chile. 
              \\
              $^b$ School of Earth and Planetary Sciences, Curtin University, Australia.
                                       \\
           $^c$ School of Mathematical Sciences,
              University of Nottingham, UK. \\[2ex]  \today}
\begin{document}
\maketitle
\begin{abstract}
In order to construct regularizations of continuous linear functionals acting on Sobolev spaces such as $W_0^{1,q}(\Omega)$, where $1<q<\infty$ and $\Omega$~is a Lipschitz domain, we propose a projection method in negative Sobolev spaces~$W^{-1,p}(\Omega)$, $p$~being the conjugate exponent satisfying $p^{-1} + q^{-1} = 1$. Our method is particularly useful when one is dealing with a rough (irregular) functional that is a member of $W^{-1,p}(\Omega)$, though not of~$L^1(\Omega)$, but one strives for a regular approximation in~$L^1(\Omega)$. 
We focus on projections onto discrete finite element spaces~$G_n$, and consider both discontinuous as well as continuous piecewise-polynomial approximations.

While the proposed method aims to compute the best approximation as measured in the negative (dual) norm, for practical reasons, we will employ a computable, discrete dual norm that supremizes over a discrete subspace~$V_m$. We show that this idea leads to a fully discrete method given by a mixed problem on~$V_m\times G_n$. We propose a discontinuous as well as a continuous lowest-order pair, prove that they are compatible, and therefore obtain quasi-optimally convergent methods. 

We present numerical experiments that compute finite element approximations to Dirac delta's and line sources. We also present adaptively generate meshes, obtained from an error representation that comes with the method. Finally, we show how the presented projection method can be used to efficiently compute numerical approximations to partial differential equations with rough data. 

\end{abstract}

\setcounter{tocdepth}{3}
\tableofcontents
\newcommand{\fintro}{f} 
%
\section{Introduction}
In the approximation of solutions to partial differential equations (PDEs), the right-hand side data (e.g., sources) may not necessarily be representable by the action of an $L^2$ (or, more generally, $L^1$) function. We will refer to such functionals as being \emph{rough}, or irregular. For instance, 
rough linear functionals $\fintro$ acting on functions $v:\Omega\to\mathbb R$, with $\Omega$ being a $d$-dimensional Lipschitz domain, include:
\begin{enumerate}[label=(\roman*)]
\item Singular actions over derivatives:
\begin{equation*}
\fintro(v):=\int_{\Omega}\vec{F}\cdot \nabla v.
\tag{where $\vec{F}$ has some kind of singularity in~$\Omega$}
\end{equation*}
\item  Point sources, defined by a Dirac delta~$\delta_{(\cdot)}$ distribution or derivatives of it:
\begin{equation*}
\fintro(v):=\langle \delta_{x_0},v\rangle=v(x_0).\tag{for a given $x_0\in\Omega$}
\end{equation*}
\item  Line sources with density $\psi$:
\begin{equation*}
\fintro(v):=\int_{\textit{C}}\psi v.\tag{for a given a contour $C\subset\overline{\Omega}$}
\end{equation*}
\end{enumerate}
\par
There are several numerical complications when dealing with rough functionals:
\begin{itemize}
\item PDEs with rough data have low-regular solutions, which imply low convergence rates with quasi-uniform discretizations (e.g., finite element discretizations using uniformly-refined meshes);
\item Adaptive methods may recover optimal convergence rates (in terms of number of degrees of freedom), however, standard refinement indicators may not be valid or may be impractical (because of the data being rough, hence does not have an $L^2$ norm)~\cite{Cohen2012,NOCHETTO2013}; 
\item Software packages may not support the implementation of rough functionals,  but only facilitate standard domain integrals, i.e.,
\begin{alignat*}{2}
 \fintro(v) = \int_\Omega \phi v\,, \tag{for a given $\phi:\Omega\rightarrow \mathbb{R}$}
\end{alignat*}
to allow for an efficient quadrature treatment.
\end{itemize}
\par
A natural idea to overcome these complications is to employ \emph{regularizations} of the rough functional~$f$; cf.~Hosseini et al.~\cite{hosseini2016regularizations}.  
To explain the effect of regularizations of~$f$ on errors, consider the abstract linear problem 
\begin{alignat*}{2}
\mathcal Lu= f \quad  \text{in } V^*
\end{alignat*}
defined by a continuous and bounded below operator $\mathcal L:U\mapsto V^*$, where $U$ and $V$ are (trial and test) Banach spaces\footnote{This is the situation commonly encountered in variational formulations of PDEs, where $b(u,v)=\left<f,v\right>_{V^*,V}$ and $\mathcal Lu:=b(u,\cdot)$, for a given continuous bilinear form $b:U\times V\to\mathbb R$.}, and $V^*$ is the dual space of~$V$. Let $f_n\in V^*$ be a regularization of $f$ and let 
\begin{alignat*}{2}
  u_n:=\mathcal L^{-1}f_n
\end{alignat*}
be the exact solution for the regularized problem.  If $u_{n,h}$ is a numerical approximation to $u_n$, then by the triangular inequality:
\begin{equation}\label{eq:err_estimate}
\|u-u_{n,h}\|_{U}\leq \underbrace{\|u-u_n\|_{U}}_{\mbox{Regularization error}}+ \underbrace{\|u_n-u_{n,h}\|_{U}}_{\mbox{Discretization error}}.
\end{equation}
Assuming that the \emph{discretization error} can be controlled efficiently by standard adaptive procedures, the error estimate~\eqref{eq:err_estimate} will be dominated by the \emph{regularization error}, for which we know that 
\begin{alignat}{2}
\label{eq:regErrorBound}
  \|u-u_n\|_{U}\leq {\gamma}^{-1}\|f-f_n \|_{V^*}\,,
\end{alignat}
where $\gamma>0$ is the stability ($\inf$-$\sup$) constant of the operator $\mathcal L$. 
Thus, the focus of attention now is on how to control the error $\|f-f_n \|_{V^*}$ (if possible, up to a given accuracy). 
Notice that the data regularization error, $f-f_n $, is naturally measured in the dual norm~$\|\cdot\|_{V^*}$, which in typical situations corresponds to a negative Sobolev space norm.
\par
The main purpose of this paper is to propose and analyse a general methodology, in the wide context of Banach spaces, to construct a robust projection of~$f$ into a finite dimensional subspace~$G_n\subset V^*$. The projection~$f_n\in G_n$ is constructed to have the desirable qualities of being regular and being a near-best approximation to~$f$ (as measured by~$\|{\cdot}\|_{V^*}$).
We focus on projections onto discrete finite element spaces~$G_n\subset L^\infty(\Omega)$, and consider both discontinuous as well as continuous piecewise-polynomial approximations. Such projections~$f_n$ allow for exact integration of the usual finite element domain integrals $\int_\Omega f_n v\,$ via quadrature.%
\footnote{Note also that when using piecewise polynomial~$f_n$, conveniently, data oscillation may vanishes in standard a~posteriori error estimates, as used in adaptive FEM; see, e.g.,~\cite{CasKreNocSieSINUM2008,CasNocIMANA2011}.}
\par
Our methodology builds upon the discrete-dual minimal-residual (DDMRes) method in Banach spaces~\cite{MugTylZeeCMAM2019,MugVdZ_SINUM2020}. The principle behind this method is residual minimization in dual norms, the idea of which can be traced back to Discontinuous Petrov--Galerkin (DPG) methods~\cite{DemGopBOOK-CH2014}. Applied to the current setting, the problem is indeed to minimize~$\|f-g_n\|_{V^*}$ amongst~$g_n\in G_n$, which is nothing but a projection problem in dual (negative) norms. For computability reasons, the dual norm is replaced by a discrete dual norm $\|f-g_n\|_{(V_m)^*}$, where~$V_m$ is a suitable discrete subspace of~$V$.
\par
The main contributions of our work are as follows. 
By means of a mathematical object known as the \emph{duality map} (see Section~\ref{theory}), we prove the equivalence between the negative-norm projection problem and a monotone-mixed formulation that is suitable for finite element discretizations (Theorem~\ref{Teo:1}). The discrete (computable) counterpart of this monotone-mixed formulation is proved to be well-posed and lead to quasi-optimal convergence (Theorem~\ref{inexact:thm}) under a Fortin compatibility condition on~$G_n \times V_m$ (cf.~\cite{GopQiuMOC2014}). In other words, the discrete method delivers projections~$\tilde f_n\in G_n$ that are near-best to~$f\in V^*$, hence satisfy:
\begin{alignat*}{2}
  \|f-\tilde f_n\|_{V^*} \le C \inf_{g_n \in G_n} \|f-g_n\|_{V^*}\,.
\end{alignat*}
Moreover, the discrete method is shown to be equivalent to a best-approximation problem in a discrete-dual norm (Theorem~\ref{Thm:inexact-method}). 
\par
We furthermore propose lowest-order pairs of finite element spaces and prove their Fortin compatibility. The~$\mathbb{P}_0/(\mathbb{P}_1 + \mathrm{bubble})$ compatible pair (Proposition~\ref{Fortin_space}) uses a discontinuous piecewise-constant finite element space for~$G_n$ and continuous linears enriched with element bubbles for~$V_m$. The~$\mathbb{P}_1/\mathbb{P}_2$ compatible pair (Proposition~\ref{prop:compatible2}) uses a continuous piecewise-linear finite element space for~$G_n$ and continuous quadratics for~$V_m$.  
\par
The discrete method also has a built-in residual representative. We show that this leads to a natural a posteriori error estimator, which can be localized and employed to conduct adaptive mesh refinements showing outstanding convergence rates (see Section~\ref{applications}). Moreover, we have observed that flatter norms (i.e., $W^{-1,p}(\Omega)$-norms with exponents $p$ closer to $1$) induce a better localization of such mesh refinements.

\subsection{Instability when using the $L^2$~projection}\label{sec:motiv_example}
We wish to highlight that a naive $L^2$~projection for rough functionals may result in unexpected or unwanted results. We illustrate this with a simple 1-D example.
\par
Let $\Omega=(0,1)$ and consider the rough functional $\fintro\in H^{-1}(0,1):=\left(H^1_0(0,1)\right)^*$ defined by:
$$
\fintro(v):=\int_0^1 x^{-{1\over4}}v'(x)\, dx, \quad\forall v\in H^1_0(0,1).
$$
For a given small parameter $\epsilon>0$, we are going to approximate this funcional using the one dimensional space generated by the \emph{hat} function:
$$
\phi_\epsilon(x):=\left\{
\begin{array}{cl}
\displaystyle {x\epsilon^{-1}} & \mbox{if } x\in(0,\epsilon),\\\\
\displaystyle {(1-x)(1-\epsilon)^{-1}} & \mbox{if } x\in(\epsilon,1).
\end{array}\right.
$$
If we intend to compute the $L^2$-projection of the rough functional $\fintro(\cdot)$ onto the one-dimensional space generated by $\phi_\epsilon$ we arrive at the problem of finding $\alpha\in\mathbb R$ such that: 
$$
\alpha \|\phi_\epsilon\|_{L^2(0,1)}^2 = \int_0^1 x^{-{1\over4}} \phi_\epsilon'(x)\,dx
={1\over\epsilon}\int_0^\epsilon x^{-{1\over4}}dx - {1\over 1-\epsilon}\int_\epsilon^1 x^{-{1\over4}}dx.
$$
Notice that the right hand side of the above equation is of order $\epsilon^{-1/4}$ and goes to infinity as $\epsilon\to 0$. However, the $L^2$-norm of $\phi_\epsilon$ equals $\sqrt{3}/3$, irrespective of $\varepsilon$. Thus, the $L^2$-projection $\alpha\phi_\epsilon$ diverges as $\epsilon\to0$.

On another hand, we have computed the exact $H^{-1}$-projection of $\fintro$ onto the one-dimensional span of $\phi_\varepsilon$, together with a discrete $H^{-1}$-projection onto the same one-dimensional space, but using our proposed methodology with a $\mathbb P_2$ test space setting (see Section~\ref{sec:P1/P2} for further details). The $L^2$-norm of these best approximations are depicted in Figure~\ref{fig:L2norm_of_H-1projection}, and compared with the divergent $L^2$-projection. The $H^{-1}$ projections show stable behaviors as $\varepsilon\to0^+$.
\begin{figure}
\centering
\includegraphics[scale=0.4]{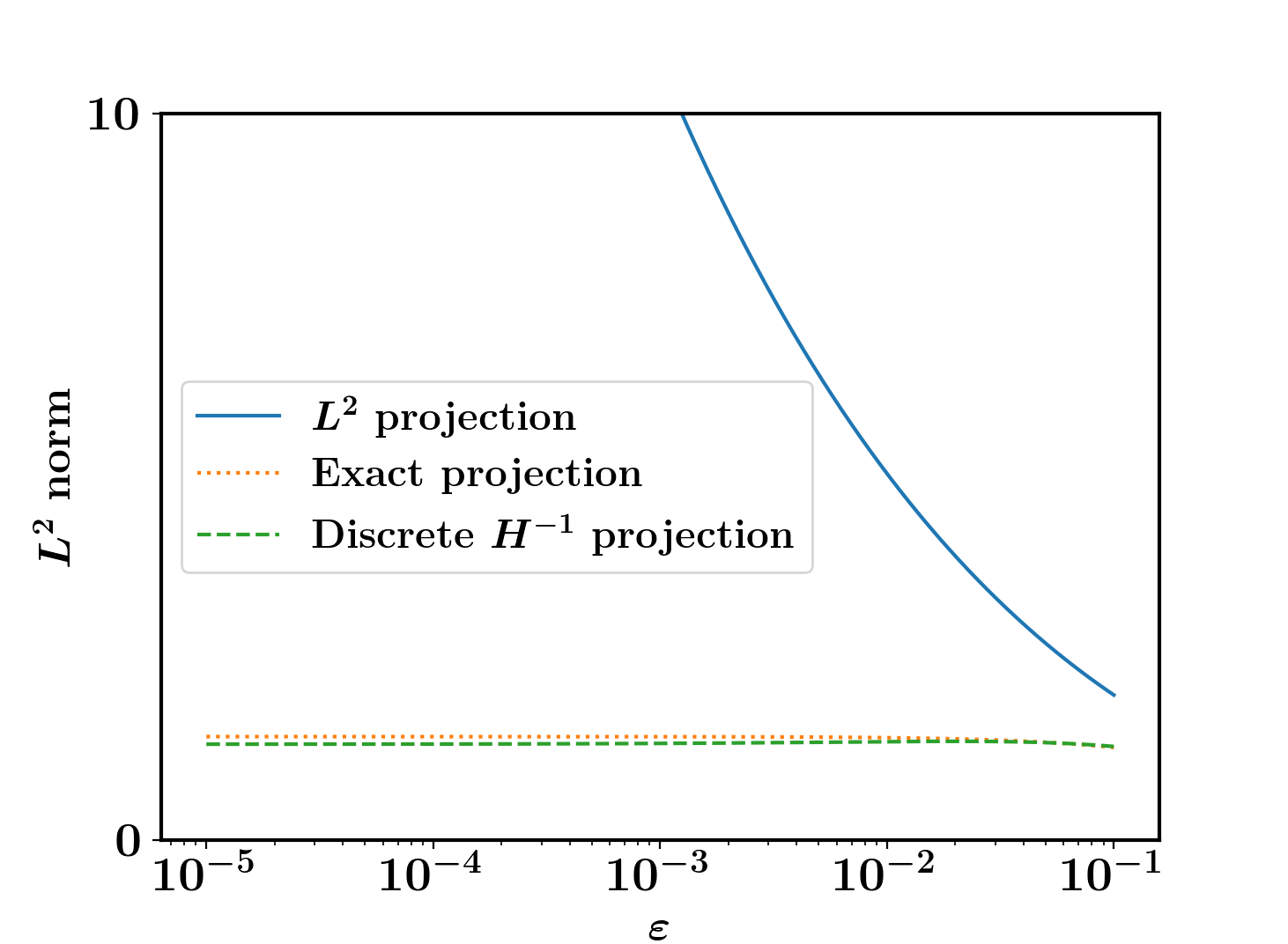}
\caption{$L^2$ norm of exact and discrete $H^{-1}$-projections of a rough functional over the one-dimensional space generated by $\phi_\epsilon$, for several small values of $\varepsilon>0$.
}
\label{fig:L2norm_of_H-1projection}
\end{figure}
\subsection{Related literature}
Solutions of second-order PDEs driven by rough sources may become singular. Indeed, it is well-known that the free-space solution of the Poisson problem:
 \begin{equation}\label{eq::poisson}
 -\Delta u=\delta_{x_0},
 \end{equation}
exhibits a $\log$-type singularity in two-dimensions; and a singular behavior of the type 
$1/\operatorname{dist}(x,x_0)$ in three-dimensions. In fact, for dimensions $d\geq 2$, the solution of~\eqref{eq::poisson} does not reach the Hilbert space $H^1(\Omega)$, mainly because the Dirac delta distribution is not in the dual space of $H^1(\Omega)$. Nevertheless, a regularized version of $\delta_{x_0}$ (e.g., an approximation represented by an $L^2(\Omega)$ function) would produce a regular solution of the Poisson problem, for which standard adaptive procedures work efficiently (see, e.g.,~\cite{Dorfler_SINUM1996,Verfurth_BOOK1996,MorNocSie_SINUM2000,BinDahDeV_NM2004,Stevenson2005,Stevenson2007}.

Rough or singular data has taken the attention of finite element analysts since the early works of Babu\v{s}ka~\cite{babuvska1971error} and Scott~\cite{scott1973finite}, where they analyzed \emph{a priori} error estimates and convergence rates for the Galerkin method applied to elliptic problems with singular source data, in the context of fractional Sobolev (Hilbert) norms $H^s$. In particular, Scott uses explicit regularizations of the delta distribution to estimate the regularization error. Later on, Eriksson \cite{eriksson1985improved} showed optimal convergence order in~$L^1$ and~$W^{1,1}$ norms, depending on adequate graded meshes adapted for Dirac delta right-hand sides. In a more general Banach-space setting, Casado-Diaz~et~al.~\cite{casado2007finite} proved the $W^{1,p}_0$-convergence and error estimates, for $1\leq p \leq{d}/({d-1})$, of piecewise linear polynomials approximating the solution of second order elliptic equations in divergence form with right-hand sides in $L^1$ (cf.~Example~\ref{ex:L1}). They also showed $W^{1,p}_0$-weak convergence when the right-hand side is a general Radon measure (see also~\cite{Clain_M3AS1995,GalHer_CRAS2004}). 

From the point of view of \textit{a posteriori} error analysis and adaptive finite element methods, rough right-hand sides in $H^{-1}$ has been considered in early works of Nochetto~\cite{NOCHETTO1995} and Stevenson~\cite{Stevenson2005,Stevenson2007}. The latter mainly rely on the approximability of $H^{-1}$ functionals by piecewise constants functions. A different approach was taken by Cohen~et~al.~\cite{Cohen2012}, where they provide $H^{-1}$-convergent algorithms directly using indicators based on local $H^{-1}$ norms. In the same spirit, Blechta~et~al.~\cite{BleMalVoh_IMAJNA2020} study the localization in negative norms $W^{-1,p}$ for \emph{a posteriori} error estimates purposes. 

Point sources have attracted major attention throughout the existing Literature. Recall that the exact solutions to these problems are commonly not encountered in standard Hilbert spaces like $H^1$. To overcome this drawback, we can distinguish adaptive approaches based on \emph{a posterior} error estimates in the natural $W^{1,p}$-setting of these equations (see, e.g.,~\cite{araya2006posteriori,araya2007adaptive,AloCamRodVal_CAMWA2014,fuica2019posteriori}); or approaches based on error estimates in fractional (and Hilbert) Sobolev norms $H^s$ (see, e.g.,~\cite{GazpozMorin2017}); or approaches based on weighted Muckenhoupt norms (see, e.g.,~\cite{agnelli2014posteriori,allendesdirac,AllOtaSal_SISC2020}). On another hand, we can also find methods based on mesh-grading techniques (see, e.g.,~\cite{ApeBenSirVex_SINUM2011,dangelo_SINUM2012}), and methods based on regularization techniques (see, e.g.,~\cite{walden1999approximation,tornberg2004numerical,SuaJacDon_SISC2014,BenVenPonTra_IJNME2014,hosseini2016regularizations,BenSym_JCP2019,HelLei_NM2020}). Many of the former results also apply for line sources (see, e.g.,~\cite{dangelo_SINUM2012,GjeKumNorWoh_M2AN2019}).

\subsection{Outline}
The outline of the paper is as follows. In Section \ref{section::2}, we present the preliminary concepts related to best approximation in Banach spaces and the functional analysis tools to be required. Section \ref{section::3} introduces the abstract regularization methodology of rough functionals as a best approximation problem (or projection problem) in dual norms, which in turns is equivalent to a monotone (non-linear) mixed method that can be approached numerically. We provide discrete well-posedness, stability, and {\it a priori} error analysis in Theorem~\ref{inexact:thm}.  Additionally,  {\it a posteriori} error analysis can be found in Theorem~\ref{thm:aposteriori}. In Section~\ref{sec:compatible_pairs}, we provide two trial/test compatible pairs for which our methodology is proved to be well-posed and stable. 
Due to their technicality, the proofs of these last results have been shifted to Appendix A and B, respectively.
In Section~\ref{applications}, we perform numerical experiment with point sources and line sources, together with solutions of PDEs with regularized source. Finally, we outline some conclusions and future work in Section~\ref{sec:conclusions}.

\section{Preliminaries}\label{section::2}

\subsection{Functional spaces and rough linear functionals}

Let  $(X,\| \cdot \|_{X})$  be  a normed space. The  dual space of $X$ (i.e., all the real valued continuous linear functionals  defined on $X$) will be denoted by $X^*$. The action of  $F\in X^*$ over  elements $x\in X$ will be denoted  by  a duality pairing  between $X^*$ and $X$, i.e.,
 $$\langle F,x\rangle_{X^*,X}:=F(x).$$
The norm of the dual space $X^*$ is defined by
\begin{equation}\label{eq::dualnorm}
 \|\cdot  \|_{X^*}:=\sup_{x\in X}\frac{\langle \cdot, x\rangle_{X^*,X}}{\|x \|_{X}}.
\end{equation}

Let $d\in\mathbb N$ denote a spatial dimension and let $\Omega\subset \mathbb{R}^d$ be a bounded Lipschitz domain.  Consider the standard Lebesgue space $L^q(\Omega)$ (for $q\geq 1$) and the Sobolev spaces:
\begin{equation}\label{Sobolev}
W^{1,q}(\Omega):= \{ v\in L^q(\Omega); \nabla v \in \left [ L^q(\Omega) \right]^d  \},
\end{equation}
 where  the $q^\text{th}$-power of the  norm is defined by
$$\|v\|^q_{W^{1,q}(\Omega)}:=\|v\|^q_{L^q(\Omega)}+\|\nabla v\|^q_{L ^q(\Omega)}.$$

Let   $\mathcal{C}^{\infty}_0(\Omega)$ be the space of smooth functions with compact support on $\Omega$, and consider  the subspace $W^{1,q}_0(\Omega)\subset W^{1,q}(\Omega)$ defined by:
\begin{equation*}
W^{1,q}_{0}(\Omega):=\overline{\mathcal{C}^\infty_0(\Omega)}^{\|\cdot\|_{W^{1,q}(\Omega)}}.
\end{equation*}
By Poincar\'e's inequality, it is well-known that $ \|\cdot\|_{W^{1,q}_0(\Omega)}:=\|\nabla (\cdot)\|_{L^q(\Omega)}$ defines an  equivalent norm on $W^{1,q}_0(\Omega)$. The dual of the Sobolev space $W^{1,q}_0(\Omega)$ will be denoted with a negative exponent, i.e.,
$$W^{-1,p}(\Omega):=\left(W_0^{1,q}(\Omega)\right)^*,$$ 
where  $p^{-1}+q^{-1}=1$. 
The associated dual norm is given accordingly to \eqref{eq::dualnorm}, in which case we talk about a  \emph{negative norm}. 
\begin{Remark}
The reader can refer to \cite{adams2003sobolev,DinPalVal_BSM2012} for the definition of Sobolev spaces $W^{s,q}(\Omega)$ with higher or even fractional derivatives of order $s>0$. In general, we will use the terminology \emph{negative norm} to denote the supremum norm~\eqref{eq::dualnorm} of the dual of a Sobolev space $W^{s,q}(\Omega)$, with $s>0$. 
\end{Remark}

The following examples, typify what we understand by irregular and rough linear functionals in negative Sobolev spaces.

\begin{exmp}[Irregular functions]\label{ex:L1}
For $f\in L^r(\Omega)$, with $r\geq 1$, the application $v\mapsto\int_\Omega fv$ defines a continuous linear functional over $L^q(\Omega)$, for any $q\geq r^*:=r/(r-1)$. However, it also defines a continuous linear functional over $W^{1,q}(\Omega)$, for a wider spectrum of values of $q\geq 1$. Indeed, one can show that 
$$\left | \int_\Omega fv \right |\leq \|f\|_{L^r(\Omega)}\|v\|_{L^{r^*}(\Omega)}\lesssim\|f\|_{L^r(\Omega)}\|v\|_{W^{1,q}(\Omega)},$$
provided the embedding $W^{1,q}(\Omega)\hookrightarrow L^{r^*}(\Omega)$ holds true continuously. This is always the case in one dimension, i.e., for $d=1$, the embedding holds true continuously for any $q\geq 1$. In higher dimensions extra assumptions are needed. For instance, if  
$\Omega\subset\mathbb R^d$ has the cone property (see \cite[Theorem~5.4]{adams2003sobolev}), then the embedding holds true continuously for any $q \geq {r^*d/ (r^*+d)}$. 
For example, if $f\in L^1(\Omega)$, but $f\notin L^{1+\epsilon}(\Omega)$, for any $\epsilon >0$, then $q\geq d$ is needed.
\end{exmp}

\begin{exmp}[Actions over derivatives]
Let  $g\in [L^r(\Omega)]^d$ with $r\geq 1$. Then the application  
\begin{equation*}
v\mapsto \int_{\Omega}g\cdot \nabla v,
\end{equation*} 
defines a continuous linear functional over $W^{1,q}(\Omega)$ for any $q\geq r^*:=r/(r-1)$.
\end{exmp}

\begin{exmp}[Point sources]\label{delta_example}
Let $\mathcal C(\Omega)$ be the space of continuous real-valued functions over $\Omega$.
For $x_0\in \Omega$, the application $v\mapsto v(x_0)$ defines a continuous linear functional over $W^{1,q}(\Omega)$ provided the embedding $W^{1,q}(\Omega)\hookrightarrow \mathcal C(\Omega)$ holds true continuously. The usual requirement for that is $q>d$ (see, e.g., Adams~\cite[Theorem~5.4]{adams2003sobolev}).
In such a case, this functional is known as the Dirac delta distribution centered at $x_0$, and we write
\begin{equation*}
\langle \delta_{x_0},v\rangle_{(W^{1,q})^*,W^{1,q}}: = v(x_0),\qquad \forall v\in W^{1,q}(\Omega).
\end{equation*}
\end{exmp}
\begin{exmp}[Line sources] 
\label{ex:line_sources}
For $d\geq 2$, let $\Gamma \subset \overline{\Omega}$ be a bounded Lipschitz curve, and  
let $\phi\in L^r(\Gamma)$. For any $q\geq \max\{d-1,r/(r-1)\}$, the linear application 
\begin{equation}\label{eq:line_source}
v\mapsto \int_\Gamma \phi v,
\end{equation}
defines a continuous linear functional over $W^{1,q}(\Omega)$. Indeed, in such a case we have 
$v\big|_\Gamma\in W^{1-{d-1\over q},q}(\Gamma)$. So the integral in~\eqref{eq:line_source} is well-defined and the whole process is continuous due to multiple applications of the trace Theorem. 
\end{exmp}

Of course, Example~\ref{ex:line_sources} can be extended to \emph{surface sources} in dimensions $d\geq 3$, and so on and so forth.

\subsection{Best approximations in Banach spaces}

For a given Banach space, the notion of projection into finite-dimensional subspaces is deeply related to the notion of \emph{best approximation}, which is formalized below in the general context of abstract normed spaces.
\begin{defn}[Best approximation]
Let  $X$ be  a normed space and consider a finite dimensional subspace $ X_n\subset X$. A \emph{best approximation} to  $x \in X $ in the finite dimensional space $X_n$, is an element $x_n\in X_n$  such that:
\begin{equation*}
\|x- x_n\|_{X} \leq \|x-z_n \|_{X},\qquad \forall z_n\in X_n.
\end{equation*}
\end{defn}

The following geometrical property of normed spaces will be necessary for the uniqueness of a best approximation (the existence of it, is due to the finite dimensionality of $X_n$; see Proposition~\ref{prop:best} below).
\begin{defn}[Strictly convex space]\label{def1:1}
A  normed  space $X$  is  strictly convex, if  for all $x_1,x_2\in X$,  such that $x_1\neq x_2$ and $\|x_1\|_X=\|x_2\|_X=1 $, it holds that:
\begin{equation*}
\|\alpha x_1+(1-\alpha)x_2\|_{X}<1,\qquad \forall \alpha \in (0,1).
\end{equation*} 
\end{defn}
\begin{Remark}
The Sobolev space $W^{1,q}(\Omega)$ defined in ~\eqref{Sobolev} is strictly convex if and only if $1<p<+\infty$. The same result holds true for the dual space $\big(W^{1,q}(\Omega)\big)^*$ (see, e.g.,~\cite{brezis2010functional,cioranescu2012geometry}).
\end{Remark}
\begin{prop}[Existence and uniqueness of a best approximation]\label{prop:best}
  Let $X$ be a Banach space and $X_n\subset X$ be a   finite-dimensional subspace. For any element in $X$, there exists at least one best approximation of it in $X_n$. 
In addition, if $X$ is a strictly convex Banach space (see Definition \ref{def1:1}), then such a best approximation is unique. Moreover, 
\begin{equation}\label{eq:rough_apriori_est}
\|x_n\|_X\leq 2\|x\|_X.
\end{equation}
\end{prop}
\begin{proof}
This is a classical result (see, e.g.,~\cite[section~10.2]{StakgoldHolds}).
\end{proof}
\begin{Remark}\label{rem:geometrical}
The estimate in~\eqref{eq:rough_apriori_est} is not sharp in general and can be improved using geometrical constants of the underlying Banach spaces (see \cite[section~3]{MugVdZ_SINUM2020}). Indeed, the reader may observe that in the Hilbert-space case, the constant in~\eqref{eq:rough_apriori_est} must be 1, since $x_n$ corresponds to the orthogonal projection of $x$.
\end{Remark}

\subsection{Duality maps}\label{theory}

The projection method that we are going to propose is based on operators called \textit{duality maps}, which allows to characterize \emph{best approximations} in a computable manner. We present a particular definition of such an operator in the context of strictly convex Banach spaces (cf.~\cite{brezis2010functional,cioranescu2012geometry}).

\begin{defn}[Duality Map]\label{dualitydef:1}
Let $X$ be a normed space and let us assume that its dual space $X^*$ is a strictly convex Banach space. For $s>1$, the duality map  $\mathcal{J}_{s,X}: X\mapsto X^* $ is the (unique)  operator satisfying:
\begin{itemize}
\item[i.] $\langle \mathcal{J}_{s,X}(x),x\rangle_{X^*,X}=\|\mathcal{J}_{s,X}(x)\|_{X^*}\|x\|_{X}$\\
\item[ii.] $\|\mathcal{J}_{s,X}(x)\|_{X^*}=\|x\|_{X}^{s-1}$.
 \end{itemize}
\end{defn}
\begin{Remark}\label{rem:duality}
The existence of the duality map $\mathcal J_{s,X}$ given in Definition~\ref{dualitydef:1} is guaranteed by the Hahn-Banach extension Theorem; while the uniqueness of it is due to the strict convexity of $X^*$ (see, e.g.~\cite{brezis2010functional,cioranescu2012geometry}). 
\end{Remark}
The following duality map identity is crucial for the characterization of \emph{best approximations} (see Corollary~\ref{coro:characterization} below).
\begin{prop}\label{dualitycharacterization}
Let $X$ be a  Banach space, such  that $X^*$ is strictly convex. Let  us consider  
 $\phi:X\rightarrow \mathbb{R}$, defined as $\phi(\cdot)=\frac{1}{s}\|\cdot\|_{X}^s$, with $s>1$. Then, $\phi$ is Gate\^{a}ux differentiable for all $x\in X$, and we have the following characterization for the duality map:
\begin{equation}\label{duality_grad}
\mathcal{J}_{s,X}(x)=\nabla{\phi}(x).
\end{equation}
\end{prop}
\begin{proof}
See e.g. \cite[chapter~1, section~2]{cioranescu2012geometry}.
\end{proof}
As a consequence of Proposition~\ref{dualitycharacterization}, we have the following Corollary.
\begin{corollary}\label{coro:characterization}
Let $X$ be a Banach space such that $X^*$ is strictly convex. Let $X_n\subset X$ be a finite dimensional subspace. If $x_n\in X_n$ is a \emph{best approximation} of $x\in X$, then by first-order optimality conditions we have: 
$$
\left<\mathcal J_{s,X}(x-x_n),z_n\right>_{X^*,X}=\left<\nabla\left({1\over s}\|x-x_n\|_X^s\right),z_n\right>_{X^*,X}= 
0, \qquad \forall z_n\in X_n.
$$ 
\end{corollary}
\begin{exmp}[Duality map of $W^{1,q}(\Omega)$]\label{duality::example1}
For $s=q>1$ and $X=W^{1,q}(\Omega)$, we have the following characterization of the duality map $\mathcal J_{W^{1,q}}:=\mathcal J_{s,X}$ (for any $v,w\in W^{1,q}(\Omega)$):
\begin{equation}\label{dualityW1p}
	\left< \mathcal{J}_{W^{1,q}}(v),w\right>_{(W^{1,q})^*,W^{1,q}}:= \int_{\Omega} |v |^{q-1}\hbox{sgn}(v)w + \sum_{i=1}^{d}\int_{\Omega} |\partial_i v|^{q-1}\hbox{sgn}\left(\partial_iv\right) \partial_i w\,.
\end{equation}
Observe that the duality map in \eqref{dualityW1p} is a non-linear operator, except for the Hilbert case $q=2$,   where the duality map coincides with the well-known Riesz map.
\end{exmp}
In general, the duality map of a dual space is difficult to compute in practice because of the supremum norm~\eqref{eq::dualnorm}. However, for smooth Banach spaces (i.e., when $X$ and $X^*$ are strictly convex and reflexive) we have the following helpful characterization.
\begin{prop}\label{lemma::duality_inverse} If $X$ and $X^*$ are strictly convex and reflexive Banach spaces, then the duality map is a bijection. Moreover, identifying $X^{**}$ with $X$, the following characterization holds true 
\begin{equation*}
\mathcal J_ {s^*,X^*} = \mathcal J_{s,X}^{-1},
\end{equation*}
where $s^*=s/(s-1)$.
\end{prop}
\begin{proof}
Reflexivity implies surjectivity of duality maps (see~\cite[chapter~II, theorem~3.4]{cioranescu2012geometry}), while strict convexity of $X$ implies injectivity of them (see ~\cite[chapter~II, theorem~1.10]{cioranescu2012geometry}). 

Let $x^*\in X^*$ and let $x=\mathcal J_{s,X}^{-1}(x^*)$. Identifying $X^{**}$ with $X$, we notice that:
$$
\left<x,x^*\right>_{X^{**},X^*}=\left<x^*,x\right>_{X^{*},X}=\left<\mathcal J_{s,X}(x),x\right>_{X^{*},X}
=\|\mathcal J_{s,X}(x)\|_{X^*}\|x\|_X=\|x^*\|_{X^*}\|x\|_{X^{**}},
$$
which implies that $x$ satisfies the first requirement of Definition~\ref{dualitydef:1}.
Moreover,
$$
\|x\|_{X^{**}}^s=\|x\|_X^s=\|x\|_X^{s^*(s-1)}=\|\mathcal{J}_{s,X}(x)\|_{X^*}^{s^*}=\|x^*\|_{X^*}^{s^*}\,,
$$
which implies that $x$ satisfies the second requirement of Definition~\ref{dualitydef:1}. Hence, by uniqueness (see Remark~\ref{rem:duality}), we must have $x=\mathcal J_ {s^*,X^*}(x^*)$.
\end{proof}

\section{The proposed projection methodology}\label{section::3}
\subsection{Exact projection in dual norms}
In this section, we establish a methodology to construct  regularizations of  functionals belonging to a dual Banach space, as the best-approximation of them over a given finite-dimensional subspace. For that, we show that the best-approximation problem is equivalent to a monotone mixed formulation, where a residual representative is introduced as a new unknown.
\begin{thm}\label{Teo:1}
 Assume that $V$ and $V^*$ are  strictly convex and reflexive Banach spaces, and  let us consider a finite dimensional subspace $G_n\subset V^*$.  Let  $\mathcal{J}_{s,V}:V\mapsto V^*$ be  the duality map of Definition~\ref{dualitydef:1}. Given $f\in V^*$, the following statements are equivalent:
\begin{itemize}
\item[i.] $f_n\in G_n$ is the unique best approximation satisfying
\begin{equation}\label{mini}
f_n=\argmin_{g_n\in G_n}\| f-g_n  \|_{V^*}.\\
\end{equation}
\item[ii.] There is a unique residual  representative $r\in V$, such that $(r,f_n)\in V\times G_n$ satisfy the semi-infinite monotone mixed formulation:
\begin{equation}\label{semi-infinite}
\left\{ \begin{array}{lrcll}
\langle \mathcal{J}_{s,V}(r),v\rangle_{V^*,V}&+\langle f_n,v\rangle_{V^*,V}&=&\langle f,v\rangle_{V^*,V},\qquad &\forall v\in V,\\
\langle g_n,r\rangle_{V^*,V}&&=&0,\qquad  &\forall g_n\in G_n.
\end{array}\right.
\end{equation}
\end{itemize}
\end{thm}
\begin{proof}
A general proof is given in \cite[Theorem 3.B]{MugVdZ_ARXIV2018} for a wider class of boundedly invertible operators $B:U\to V^*$ (where $U$ is another Banach space), but using the particular choice of duality map 
$\mathcal J_{2,V}$. It is straightforward to accommodate that proof to the case where $U=V^*$, $B$ is the identity operator in $V^*$, and $\mathcal J_{s,V}$ is any duality map fulfilling Definition~\ref{dualitydef:1}. Indeed, just for illustrating we will give a proof of how~\eqref{mini} implies~\eqref{semi-infinite}.

Let $f_n$ be the best-approximation satisfying~\eqref{mini}, which is guaranteed by Proposition~\ref{prop:best}. 
Consider the duality map $\mathcal J_{s^*,V^*}:V^*\mapsto V^{**}$, where $s^*=s/(s-1)$. By Corollary~\ref{coro:characterization} and Proposition~\ref{lemma::duality_inverse}, we have:
\begin{equation}\label{thm::eq:1}
0= \langle \mathcal J_{s^*,V^*}(f-f_n), g_n \rangle_{V^{**},V^*}=\langle g_n,\mathcal J^{-1}_{s,V} (f-f_n)\rangle_{V^*,V} \,, \qquad \forall g_n\in G_n.
\end{equation}
Defining the variable $r:=  \mathcal J^{-1}_{s,V} (f-f_n) \in V$ and plugging it into eq.~\eqref{thm::eq:1} we obtain the second equation of the mixed system~\eqref{semi-infinite}. Moreover, since 
$\mathcal J_{s,V} (r)=f-f_n \in V^*$, we also obtain the first equation of~\eqref{semi-infinite}.
\end{proof}
\begin{Remark}
Using the definition of the duality map (see Definition \ref{dualitydef:1}), we get the following relation between the residual representative and the \emph{best-approximation error}:
\begin{equation}\label{residualestimation}
\|f-f_n\|_{V^*}=\|r\|^{s-1}_{V},\qquad \mbox{ for } s>1.
\end{equation}
Notice that the residual representative $r= \mathcal J^{-1}_{s,V} (f-f_n)$ depends on the choice of the duality map (i.e., the parameter $s>1$), while the best-approximation $f_n$ is independent of that choice.
\end{Remark}
\subsection{The fully-discrete practical method}
The monotone mixed formulation~\eqref{semi-infinite} is still intractable for computational purposes unless $V$ has finite dimension. The standard way to overcome this drawback is to consider a finite dimensional subspace  $V_m\subset V$ and try to compute the following fully-discrete mixed problem:
\begin{equation}\label{fully1}
\left\{ \begin{array}{lll}
\mbox{Find $(r_m,\tilde f_n)\in V_m\times G_n$ such that}\\
\langle \mathcal{J}_{s,V}(r_m),v_m\rangle_{V^*,V}+\langle \tilde f_n,v_m\rangle_{V^*,V} & =  \langle f,v_m\rangle_{V^*,V}\,, & \quad\forall v_m\in V_m,\\
\langle g_n,r_m\rangle_{V^*,V} & = 0, & \quad\forall g_n\in G_n.
\end{array}\right.
\end{equation}
Observe that we have used the notation $\tilde f_n$ to distinguish between the solution of~\eqref{fully1} and the solution $f_n$ of the semi-infinite mixed system~\eqref{semi-infinite}, or equivalently, the best-approximation~\eqref{mini}.\\

Of course, many questions arise now:
\begin{itemize}
\item Is also $\tilde f_n\in G_n$ a best-approximation to $f$ in some sense?
\item Is the fully discrete mixed problem~\eqref{fully1} well-posed?
\item Is the solution $\tilde f_n$ quasi-optimal in the sense that $\|f-\tilde f_n\|_{V^*}\lesssim \|f-f_n\|_{V^*}$?
\item Is possible to use $\|r_m\|_V$ as a reliable and efficient error estimate to drive adaptivity?
\end{itemize}
The answer to these queries will guide the following theorems.

\begin{thm}\label{Thm:inexact-method} 
 Assume that $V$ and $V^*$ are strictly convex and reflexive Banach spaces.  Let $f\in V^*$ and consider finite dimensional approximation spaces $G_n\subset V^*$ and $V_m \subset V$. 
A discrete functional $\tilde f_n\in G_n$ solves the fully-discrete mixed system~\eqref{fully1} (together with $r_m\in V_m$), if and only if, $\tilde f_n$ is a best-approximation to $f\in V^*$ in the following sense:
\begin{equation}\label{eq:discrete_min}
\tilde f_n=\argmin_{g_n\in G_n} \| f-g_n\|_{(V_m)^*}\,,\quad \mbox{ where } 
\|\cdot\|_{(V_m)^*}:=\sup_{v_m\in V_m}\frac{\langle \,\cdot\,, v_m\rangle_{V* ,V}}{\|v_m\|_{V}}.
\end{equation}
\end{thm}
\begin{proof}
See \cite[Theorem~4.1]{MugVdZ_SINUM2020}.
\end{proof}
\begin{Remark}
Observe that the solution of~\eqref{eq:discrete_min} may not be unique, even when $(V_m)^*$ is strictly convex. 
This is because $\|\cdot\|_{(V_m)^*}$ is indeed a norm in $(V_m)^*$, but it is not a norm in $V^*$. In particular, two different elements of $G_n$ may have the same action over the elements of $V_m$. The following Theorem~\ref{inexact:thm} provides a sufficient condition to guarantee the well-posedness of~\eqref{fully1}, or equivalently~\eqref{eq:discrete_min}.
\end{Remark}

\begin{thm}\label{inexact:thm}
Let $V$ and  $V^*$ be strictly convex and reflexive Banach spaces. Assume that the finite dimensional approximation subspaces $G_n\subset V^*$ and $V_m \subset V$ satisfy the existence of a continuous (Fortin) operator 
$\Pi:  V\rightarrow V_m$  such that:
\begin{itemize}
\item[i.] 
\itemEq{\| \Pi v \|_{V}\leq C_{\Pi}\|v\|_{V},\quad \forall v\in V\mbox{ and some } C_{\Pi} >0.
\label{constC}}
\item[ii.] \itemEq{\langle g_n, v-\Pi v\rangle_{V^*,V}=0, \quad\forall g_n \in G_n,\forall v\in V.\label{Fortin:c3}}
\end{itemize}
Then, for any $f\in V^*$, there is a unique $(r_m,\tilde f_n)\in V_m\times G_n$ solution of problem~\eqref{fully1}. 
The solution satisfies the apriori estimates:
\begin{equation}\label{eq:apriori}
\|r_m\|_V^{s-1}\leq \|f\|_{V^*}.
\qquad\mbox{ and }\qquad \|\tilde f_n\|_{V^*}\leq 2C_\Pi \|f\|_{V^*}.
\end{equation}
Moreover, recalling the solution of~\eqref{mini} $f_n\in G_n$, we have the quasi-optimality properties:
\begin{alignat}{2}
\|r_m\|_V^{s-1} & \leq\inf_{g_n\in G_n}\|f-g_n\|_{V^*}=\|f-f_n\|_{V^*}\label{eq:quasi1}\\
\|f-\tilde f_n\|_{V^*} & \leq (1+2C_\Pi)\inf_{g_n\in G_n}\|f-g_n\|_{V^*}= (1+2C_\Pi)\|f-f_n\|_{V^*}.\label{eq:quasi2}
\end{alignat}
\end{thm}
\begin{proof}
A general well-posedness proof can be found in \cite[Theorem~4.5]{MugVdZ_SINUM2020} (just accommodate it considering the operator $B$ as the identity in $V^*$). Nevertheless, we will show here how to obtain the estimates 
\eqref{eq:apriori}, \eqref{eq:quasi1} and \eqref{eq:quasi2}, since their proof is slightly different. Indeed, testing the first equation of the fully-discrete mixed problem~\eqref{fully1} with $v_m=r_m$, using the ortogonal property of $r_m$, and the definition of the duality map (see Definition~\ref{dualitydef:1}), we obtain:
$$
\|r_m\|^s_{V}=\left<f,r_m\right>_{V^*,V}=\left<f-g_n,r_m\right>_{V^*,V}\,\quad \forall g_n\in G_n,
$$
which gives the first estimate in~\eqref{eq:apriori} and also~\eqref{eq:quasi1} after using Cauchy-Schwarz's inequality. For the second estimate in~\eqref{eq:apriori} observe that:
\begin{alignat}{2}
\|\tilde f_n\|_{V^*} = \sup_{v\in V}{\left<\tilde f_n,v\right>_{V^*,V}\over \|v\|_V}
\leq & C_\Pi\sup_{v\in V}{\left<\tilde f_n,\Pi v\right>_{V^*,V}\over \|\Pi v\|_V}\tag{by~\eqref{constC}~and~\eqref{Fortin:c3}}\\
\leq & C_\Pi \|\tilde f_n\|_{(V_m)^*}\tag{since $\Pi V\subset V_m$}\\
\leq & 2 C_\Pi \|f\|_{(V_m)^*}\tag{by~\eqref{eq:discrete_min}~and~\eqref{eq:rough_apriori_est}}\\
\leq & 2 C_\Pi \|f\|_{V^*}\,.\tag{since $V_m\subset V$}
\end{alignat}
Moreover, it is easy to see that the application $P_n:V^*\to G_n$ such that $P_n(f):=\tilde f_n$ defines a projector for which $\|P_n(f)\|_{V^*}\leq 2C_\Pi \|f\|_{V^*}$ and $P_n(f-g_n)=P_n(f)-g_n$, for any $g_n\in G_n$. Hence we have:
$$
\|f-\tilde f_n\|_{V^*}=\|(I-P_n)f\|_{V^*}=\|(I-P_n)(f-g_n)\|_{V^*}\leq (1+2C_\Pi)\|f-g_n\|_{V^*}\,,
$$
which proves~\eqref{eq:quasi2}.
\end{proof}
\begin{Remark} An operator satisfying \eqref{constC} and \eqref{Fortin:c3} is known as a Fortin operator (see~ \cite{boffi2013mixed}). The existence of such a Fortin operator requires that 
\begin{equation}\label{forti_condition}
\dim( G_n)\leq \dim( V_m).
\end{equation}
Observe that to ensure stability and quasi-optimality, the constant $C_\Pi>0$ must be uniformly bounded in terms of the discretization parameters $\{n,m\}$ of the underlying discrete spaces $G_n$ and $V_m$. 
\end{Remark}
\begin{Remark}
The stability constants $2C_\Pi$ in~\eqref{eq:apriori} and $(1+2C_\Pi)$ in~\eqref{eq:quasi2} are not sharp in general. They can be improved using geometrical constants of the underlying Banach spaces $V$ and $V^*$. See~\cite[section~4.4]{MugVdZ_SINUM2020} for the details. 
\end{Remark}
\begin{Remark}\label{rem:rates}
For finite element discretizations on quasi uniform meshes $\{\mathcal T_h\}_{h>0}$, one would expect that the best approximation $\|f-f_n\|_{V^*}$ is bounded by a constant times $h^\tau$, where $\tau>0$ is limited by the regularity of $f\in V^*$ and the polynomial degree of the finite element space.
See Section~\ref{applications} for examples with $V^* = W^{-1,p}$.
\end{Remark}

\subsection{A posteriori error estimate}
In residual minimization methods, it is customary to use the quantity $\|r_m\|_V$ as an error estimate to drive adaptivity procedures. The next Theorem aims to answer the query about if $\|r_m\|_V$, as an a posteriori error estimate, is indeed reliable and efficient.

\begin{thm}[A posteriori error estimator]\label{thm:aposteriori}
Assume the same conditions of Theorem~\ref{inexact:thm}. For any $f\in V^*$, the counterpart $r_m\in V_m$ of the unique solution of the discrete problem~\eqref{fully1} satisfies:
\begin{equation}\label{eq:aposteriori}
\|r_m\|^{s-1}_V\leq \|f-\tilde f_n\|_{V^*}\leq \operatorname{osc}(f) + C_\Pi\|r_m\|_V^{s-1},
\end{equation}
where the oscillation term is defined by
$$
\operatorname{osc}(f):=\sup_{v\in V}{\left< f,v-\Pi v \right>_{V^*,V}\over \|v\|_V}.
$$
\end{thm}
\begin{proof}
The first inequality (from left to right) in~\eqref{eq:aposteriori} is an immediate consequence of~\eqref{eq:quasi1}. For the second inequality observe that:
\begin{alignat}{2}
\|f-\tilde f_n\|_{V^*} 
= &  \sup_{v\in V}{\left<f-\tilde f_n,v-\Pi v +\Pi v\right>_{V^*,V}\over \|v\|_V}\tag{since $-\Pi v+\Pi v=0$}\\
\leq & \operatorname{osc}(f)+C_\Pi \sup_{v\in V}{\left<f-\tilde f_n,\Pi v\right>_{V^*,V}\over \|\Pi v\|_V}\tag{by~\eqref{constC}~and~\eqref{Fortin:c3}}\\
\leq & \operatorname{osc}(f)+C_\Pi \sup_{v\in V}{\left<\mathcal J_{s,V}(r_m),\Pi v\right>_{V^*,V}\over \|\Pi v\|_V}\tag{by~\eqref{fully1}}\\
\leq & \operatorname{osc}(f)+C_\Pi \|r_m\|^{s-1}_V,\notag
\end{alignat} 
where the last inequality has been obtained using Cauchy-Schwarz's inequality and Definition~\ref{dualitydef:1}.
\end{proof}
\begin{Remark}
Using property~\eqref{Fortin:c3} observe that: 
$$
\operatorname{osc}(f)=\sup_{v\in V}{\left< f-g_n,v-\Pi v \right>_{V^*,V}\over \|v\|_V},\quad\forall g_n\in G_n.
$$
Hence, $\operatorname{osc}(f)\leq (1+C_\Pi)\inf_{g_n\in G_n}\|f-g_n\|_{V^*}$, which combined with~\eqref{eq:aposteriori} and~\eqref{eq:quasi1} gives another way to prove~\eqref{eq:quasi2}. 
\end{Remark}

\section{Compatible pairs}\label{sec:compatible_pairs}
In this section we introduce two practical options of compatible pairs $G_n$-$V_m$ verifying the requirements of Theorem~\ref{inexact:thm}. 
The functional context is the following. Let us consider a bounded Lipschitz domain $\Omega\subset\mathbb R^d$, and $V:=W^{1,q}_0(\Omega)$ with $q>d$.\footnote{This last requirement allows us the use of the Lagrange interpolant~\cite[section~1.5.1]{ern2013theory} in the proofs of Propositions~\ref{Fortin_space} and~\ref{prop:compatible2}. The results may be extended to the whole range of $q\geq 1$ using the Scott-Zhang interpolant~\cite[section~1.6.2]{ern2013theory}. However, the proofs would become more technical than they already are.}
Let $V^*=W^{-1,p}(\Omega)$ be the dual space of $V$, where $p=q/(q-1)$, 
and let $\mathcal{T}_h=\{T_i\}_{i=1}^{n}\subset \Omega$ be a simplicial partition of disjoint open elements such that $\cup_{i=1}^n\overline{T_i}=\overline{\Omega}$.

\subsection{The $\mathbb P_0/(\mathbb P_1 +\mbox{bubbles})$ compatible pair}
Let 
\begin{equation}\label{trialspace:1}
\left\{\begin{array}{l}
G_n:=\text{span}\{\mathcal{G}_1,...,\mathcal G_n\} \subset V^*,  \\
\text{where, }  \,
\langle \mathcal{G}_{i},\phi\rangle_{V^*,V}:=\displaystyle\int_{T_i}\phi, \quad \forall \phi \in V,
\mbox{ for each } T_i\in\mathcal T_h.
\end{array}\right.
\end{equation}
The space $G_n$ defined above is an analog of the piecewise constant space $\mathbb P_0$. However, notice that $G_n$ is a space of functionals or \emph{actions}, instead of space of \emph{functions}. 
In order to solve the mixed system~\eqref{fully1}, we need to come up with a discrete test space $V_m\subset V$ satisfying the requirements of Theorem~\ref{inexact:thm}. For that, we consider the interior local bubble functions $\mathit{b}_{i}\in W^{1,q}_0(T_i)$ defined by: %
\begin{equation}\label{eq:bubble}
b_i(x)=\prod^{d+1}_{j=1} \lambda_j(x),\qquad \forall  i=1,...,n,
\end{equation}
where $\{\lambda_j\}$ are the barycentric coordinates of the simplex $T_i$.
The $n$-dimensional space generated by these bubble functions will be denoted by: 
\begin{equation}\label{Bubble}
\mathbb {B}_n(\mathcal{T}_h):=\left\{v\in W^{1,q}_0(\Omega)\cap \mathcal{C}(\overline{\Omega}); v|_{T_i} \in \text{span}\{\mathit{b}_{i}\}, \forall T_i\in \mathcal{T}_h \right\}.
\end{equation}
Additionally, we consider the piecewise polynomial finite element space
\begin{equation}\label{eq:P1} 
\mathbb{P}_1(\mathcal{T}_h):=\left\{v\in \mathcal{C}(\overline{\Omega}); v|_{T_i}\in \mathbb{P}_1, \forall T_i\in \mathcal{T}_h\right\}.
\end{equation}
\begin{prop}\label{Fortin_space}Assume that we have a shape-regular family of affine simplicial meshes $\{\mathcal T_h\}_{h>0}$.
If $V_m\subset V$ is a finite dimensional subspace containing the spaces $\mathbb{B}_n(\mathcal{T}_h)$ and $\mathbb{P}_1(\mathcal{T}_h)\cap W^{1,q}_0(\Omega)$, then $V_m$ and $G_n$ (defined in~\eqref{trialspace:1}) satisfy 
the assumptions of Theorem~\ref{inexact:thm}, i.e., 
there exists a Fortin operator $\Pi:V  \mapsto V_m$ verifying~\eqref{constC} and~\eqref{Fortin:c3}.
\end{prop}
\begin{proof}
See Appendix~\ref{appendix}.
\end{proof}
\begin{Remark}
An alternative to $\mathbb{B}_n(\mathcal{T}_h)$ can be any $n$-dimensional space generated by piecewise linear and continuous bubbles supported on each of the elements of $\mathcal{T}_h$, which somehow is a space of extra $h$-refinements of $\mathbb P_1(\mathcal T_h)$. 
\end{Remark}
\begin{Remark}\label{rem:piecewise_poly}
Notice that the following practical piecewise polynomial finite element space contains both $\mathbb{B}_n(\mathcal{T}_h)$ and $\mathbb{P}_1(\mathcal{T}_h)\cap W^{1,q}_0(\Omega)$ spaces:
$$
\mathbb{P}_{d+1}(\mathcal{T}_h)\cap W^{1,q}_0(\Omega):=\left\{v\in W^{1,q}_0(\Omega)\cap \mathcal{C}(\overline{\Omega}); v|_{T_i}\in \mathbb{P}_{d+1}, \forall T_i\in \mathcal{T}_h\right\}.
$$
\end{Remark}
%
\subsection{The $\mathbb P_1/\mathbb P_2$ compatible pair}\label{sec:P1/P2}
Consider the space $\mathbb P_1(\mathcal T_h)$ defined in~\eqref{eq:P1} and let $\{\varphi_i\}_{i=1}^{N_v}$ be the set of nodal basis functions spanning $\mathbb P_1(\mathcal T_h)\cap W_0^{1,p}(\Omega)$, where $N_v$ corresponds to the number of interior vertices associated with $\mathcal T_h$. Let
\begin{equation}\label{eq:G_nv}
\left\{\begin{array}{l}
G_{N_v}:=\text{span}\{\mathcal{G}_1,...,\mathcal G_{N_v}\} \subset V^*,  \\
\text{where, } \,
\langle \mathcal{G}_{i},\phi\rangle_{V^*,V}:=\displaystyle\int_{\Omega}\varphi_i\,\phi, \quad \forall \phi \in V,
\mbox{ for each } i=1,...,N_v.
\end{array}\right.
\end{equation}
Moreover, let
\begin{equation}\label{eq:P2} 
\mathbb{P}_2(\mathcal{T}_h):=\left\{v\in \mathcal{C}(\overline{\Omega}); v|_{T_i}\in \mathbb{P}_2, \forall T_i\in \mathcal{T}_h\right\}.
\end{equation}
The next proposition establishes the compatibility of a space $V_m\supseteq\mathbb{P}_2(\mathcal{T}_h)\cap W^{1,q}_0(\Omega)$ with $G_{N_v}$, under the following mesh assumption.
\begin{assumption}[Quasi-uniform patches]\label{ass:mesh} Let $\{\mathcal T_h\}_{h>0}$ be a shape-regular family of affine simplicial meshes. Let $\{\varphi_i\}_{i=1}^{N_v}$ be the set of nodal basis functions spanning $\mathbb P_1(\mathcal T_h)\cap W_0^{1,p}(\Omega)$. For each $i=1,...,N_v$, let $P_i:=\operatorname{supp}\varphi_i$ be the patch of elements supporting the function $\varphi_i$. Let $h_T>0$ denote the diameter of an element $T \subset P_i$ and let $h_i=\max_{T\subset P_i}h_T$.
We assume the existence of a mesh-independent constant $c>0$ such that $h_T\leq c h_i$, for every element $T\subset P_i$. 
We further assume that for every patch $P_i$, there is a reference patch $\widehat P_i\subset\mathbb R^d$ of unitary measure, such that every element $\widehat T\subset \widehat P_i$ is mapped onto a unique element $T\subset P_i$ through an affine transformation $F_T:\widehat T\to T$ of the form $F_T(\widehat x)= A_T\widehat x + y_T$, where:
\begin{equation}\label{eq:scaling}
|\det A_T| = {|T|\over|\widehat T|} =:\eta_i,\qquad\forall T\subset P_i.
\end{equation}
The constant $\eta_i$ will be referred to as the \emph{scaling constant} of the patch $P_i$.
\end{assumption}

\begin{prop}\label{prop:compatible2}
Under the hypothesis of Assumtion~\ref{ass:mesh}, if $V_m\subset V$ is a finite dimensional subspace containing the space $\mathbb P_2(\mathcal T_h)\cap W^{1,q}_0(\Omega)$, then $V_m$ and $G_{N_v}$ (defined in~\eqref{eq:G_nv}) satisfy 
the assumptions of Theorem~\ref{inexact:thm}, i.e., 
there exists a Fortin operator $\Pi:V  \mapsto V_m$ verifying~\eqref{constC} and~\eqref{Fortin:c3}.
\end{prop}
\begin{proof}
See Appendix~\ref{appendix2}.
\end{proof}

\section{Applications}\label{applications}

\subsection{Point sources.}

As a first application,  we consider projections of \emph{Dirac delta} distributions (point sources, see Example~\ref{delta_example}). It is well known that this distribution  does not belong to the Hilbert space $H^{-1}(\Omega):=W_0^{-1,2}(\Omega)$ for dimensions higher or equal than two. 
In our case, we will consider standard Sobolev spaces\footnote{i.e., of integer order and without weighted norms.} in which the action of the \emph{Dirac delta} is linear and continuous as it was mentioned in Example \ref{delta_example}.

\subsubsection{One dimensional Dirac's Delta projection}
Given a partition $\mathcal{T}_h=\{T_i\}_{i=1}^{n}$ of $\Omega:=(0,1)$, we consider the trial spaces $G_n$ and $G_{N_v}$ defined in~\eqref{trialspace:1} and~\eqref{eq:G_nv}, together with the test space $V_m=\mathbb P_2(\mathcal{T}_h)\cap W_0^{1,q}(\Omega)$.%
We compute the mixed system~\eqref{fully1}, using the duality map related with the norm $\|\cdot\|_{W_0^{1,q}}$, i.e., 
$$
\left<\mathcal J_{W_0^{1,q}}(v),w\right>_{W^{-1,p},W_0^{1,q}}:=\sum_{i=1}^{d}\int_{\Omega} |\partial_i v|^{q-1}\hbox{sgn}\left(\partial_iv\right) \partial_i w,
\quad\forall v,w\in W_0^{1,q}(\Omega).
$$
For $p=q=2$ the duality map is linear. In Fig.~\ref{delta_projection_unif} we represent graphically the projections obtained for $\delta_{x_0}$ in such a case, with $x_0=0.5$ and the trial space $G_{N_v}$. 
We have considered uniform meshes of $n=16, 32$~\&~$64$ elements respectively.  Results are coherent with what is expected (cf.~\cite{hosseini2016regularizations}). 
\begin{figure}[htb]
\centering
\includegraphics[width=0.32\textwidth]{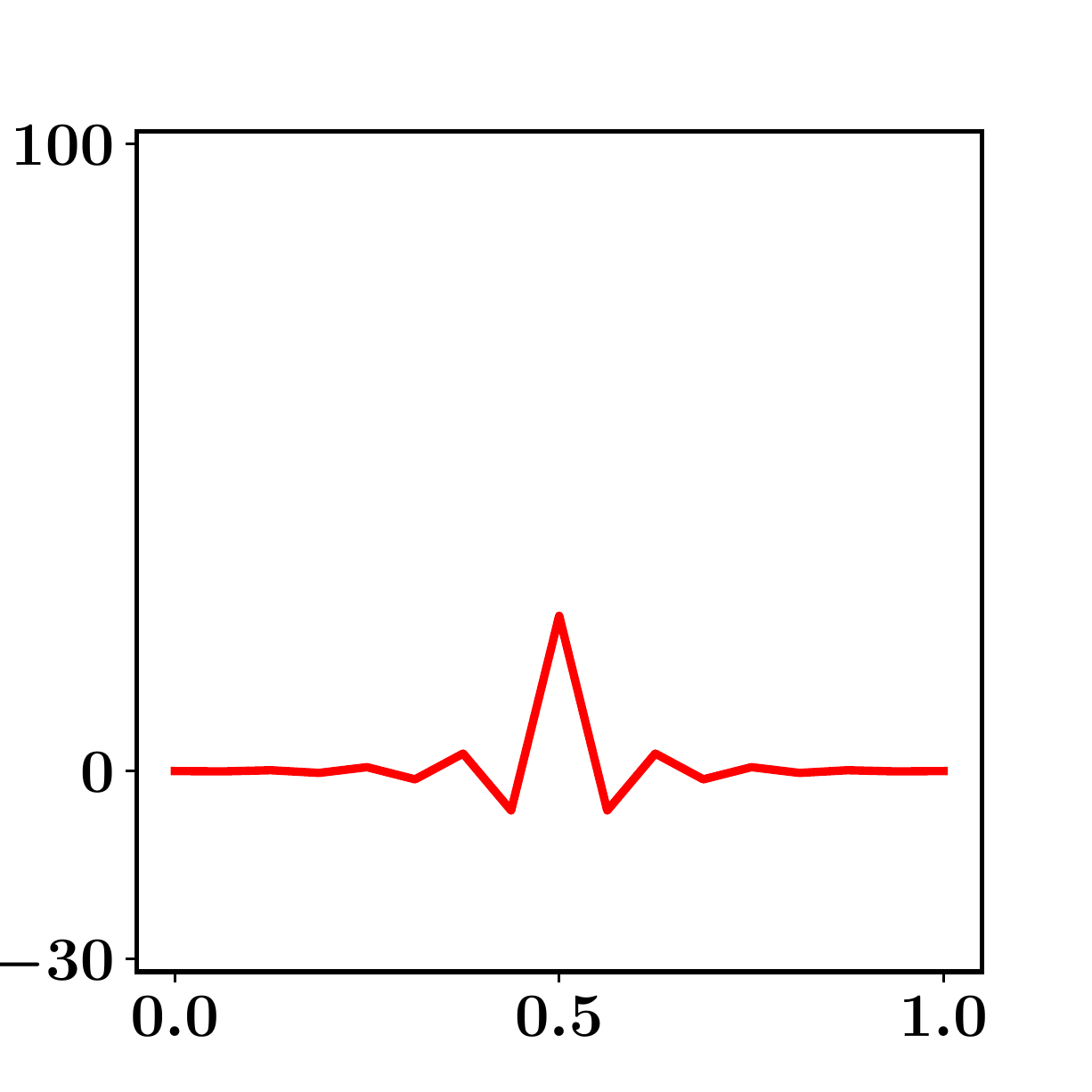}
\includegraphics[width=0.32\textwidth]{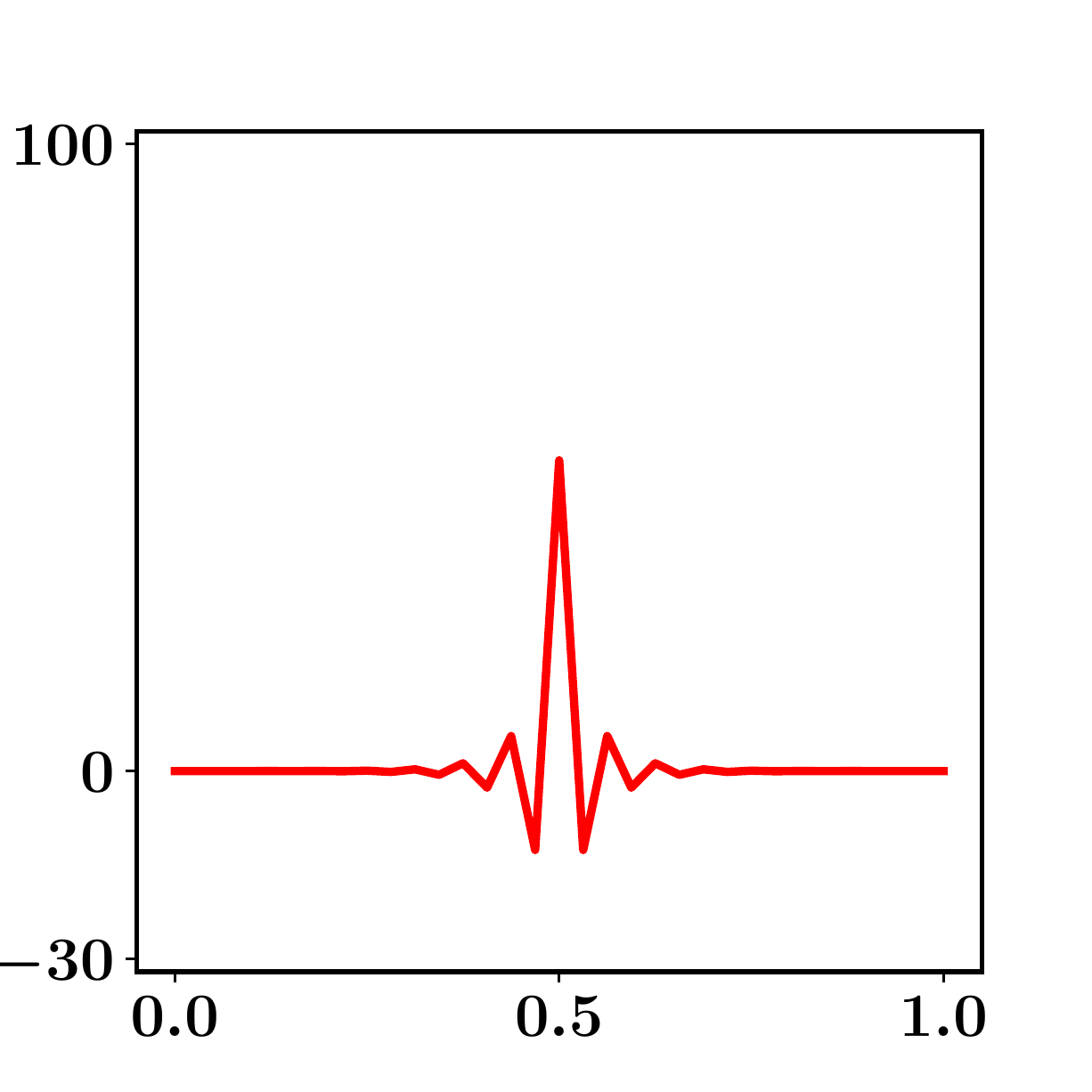}
\includegraphics[width=0.32\textwidth]{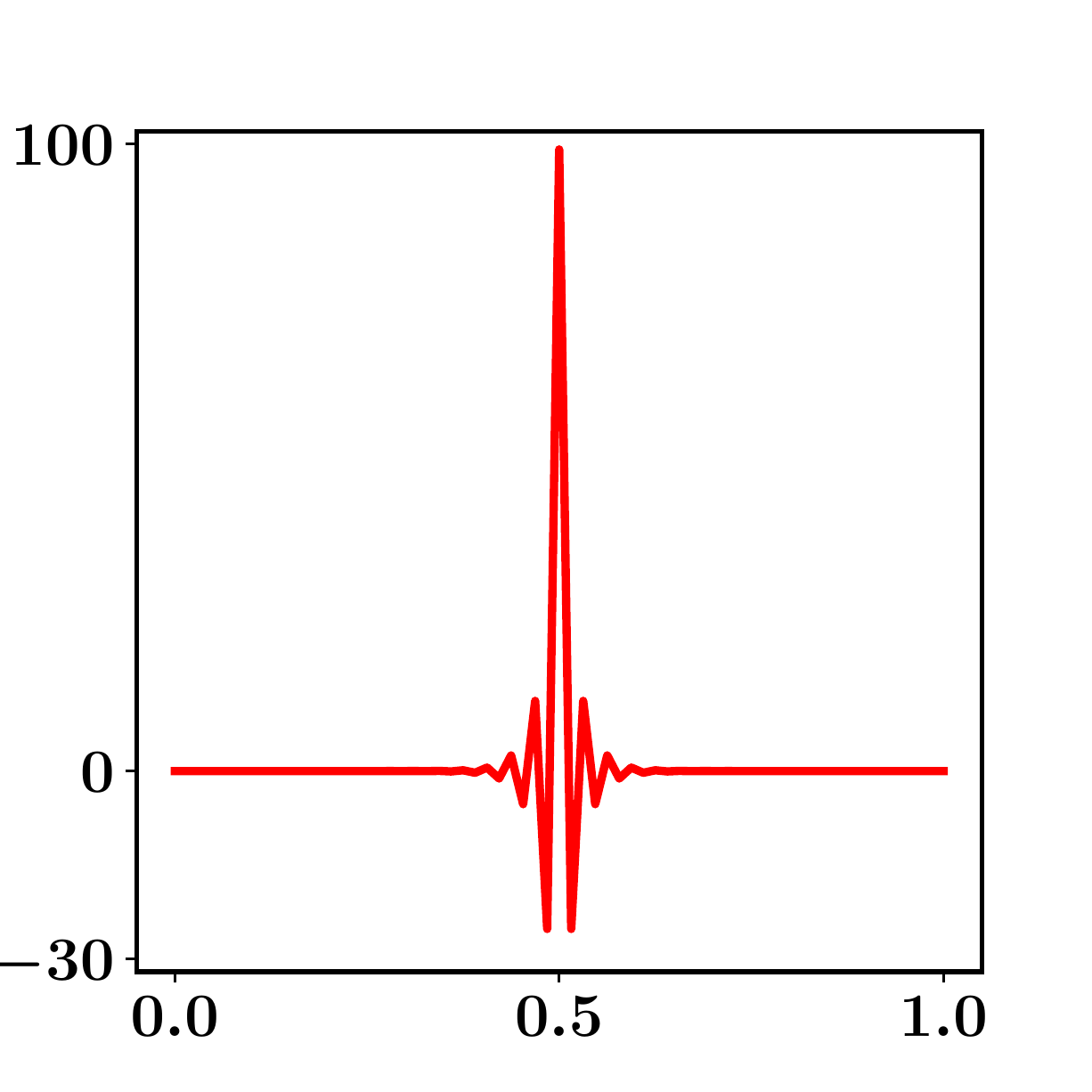}
\caption{Sequence of Dirac delta projections over uniform meshes of $n=16,32$ \& $64$ elements.}
\label{delta_projection_unif}
\end{figure}

For $p={q\over q-1}<2$, the duality map is nonlinear. Hence, we have resorted to a Newton-Raphson continuation method to solve numerically problem  \eqref{fully1}.  That is, we create a sequence of $(k+1)$ problems with parameters $2=p_0>\dots>p_{k-1}>p$, where in each step, problem~\eqref{fully1} is solved using the solution of the previous step as initial guess. 
Using the piecewise constant trial space $G_n$ defined in~\eqref{trialspace:1}, Fig.~\ref{uniform_error} depicts the convergence of the  residual  term  $\|r_m \|^{q-1}_{V}$ compared with total degrees of freedoms (i.e., $\mbox{dim}(G_n)+ \mbox{dim}(V_m)$) for uniform and adaptive $h$-refinements, and for several values of $p\in(1,2)$. Recall that, by Sobolev embeddings in 1D, the Dirac delta action is well-defined in $W_0^{s,q}(\Omega)\subset\mathcal C(\Omega)$ whenever $sq>1$ (see, e.g.,~\cite{adams2003sobolev}). The observed convergence rate for $h$-refinements is $1/p$, which can be seen as the difference between the regularity exponent $s=1$ and the critical regularity exponent $s^*=1/q$ (cf.~Remark~\ref{rem:rates}). On the other hand, since the source localizes in only one point, exponential convergence rates are observed for adaptive $h$-refinements.
The marking criteria has been set to refine all the elements showing local error larger than the 50\% of the maximum local error.
\begin{figure}[htb]
\centering
\includegraphics[width=0.7\textwidth]{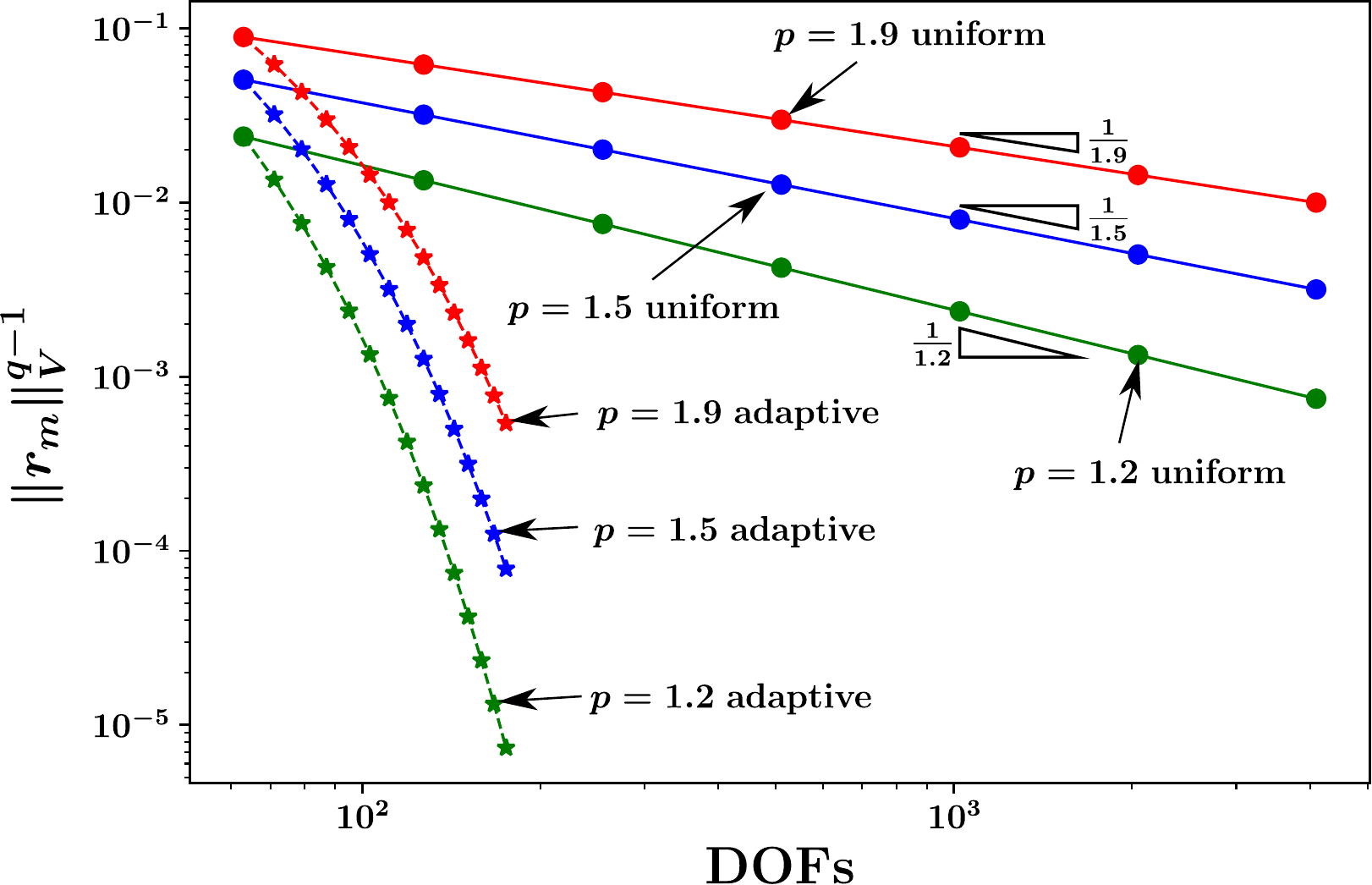}
\caption{1D Dirac delta: convergence rates of uniform and adaptive $h$-refinements for several values of $p$.}%
\label{uniform_error}
\end{figure}

\subsubsection{Elliptic ODE with projected Dirac delta source}

In this section, we test the performance of the projected Dirac delta acting as the source term of an elliptic ODE. Let $x_0\in\Omega:=(0,1)$ and $a\geq 0$. Consider the following {\it exact} problem:
\begin{equation}\label{ODE}
\left\{ \begin{array}{cl}
-u^{''}+a^2u&=\delta_{x_0}  \quad \mbox{ in } \Omega,\\
u(0)=u(1)&=0.
\end{array}\right.
\end{equation}
It is easy to check that the analytical solution of~\eqref{ODE} is:
\begin{equation}\label{exact:1}
\displaystyle u(x)=
\left\{
\begin{array}{ll}\displaystyle
{e^{ax_0}-e^{a(2-x_0)}\over 1-e^{2a}}{\sinh(ax)\over a} & \quad \hbox{ if }   0 \leq x \leq x_0,  \\\\\displaystyle
{\sinh(ax_0)\over a} {e^{ax}-e^{a(2-x)}\over 1-e^{2a}}  & \quad \hbox{ if }  x_0 \leq x \leq 1.
\end{array}
\right.
\end{equation}
The case $a=0$ can be obtained performing the limit when ${a\to 0}$ in~\eqref{exact:1}, in whose occurrence the solution is piecewise linear and continuous (see Figure~\ref{fig:ODE_uniform}).
Let $\delta_n\in G_n$ be the piecewise constant projection of the Dirac delta, computed using the $\mathbb P_0/\mathbb P_2$ compatible pair, with $p=q=2$. Observe that $\delta_n\in L^2(\Omega)$, which induces the following regularized problem: 
\begin{equation}\label{ODE:REG}
\left\{ \begin{array}{cl}
-u_n^{''}+a^2u_n&=\delta_n, \quad \mbox{ in } \Omega, \\
u_n(0)=u_n(1)&=0.
\end{array}\right.
\end{equation}
\begin{figure}[htb]
\centering
\includegraphics[width=0.32\textwidth]{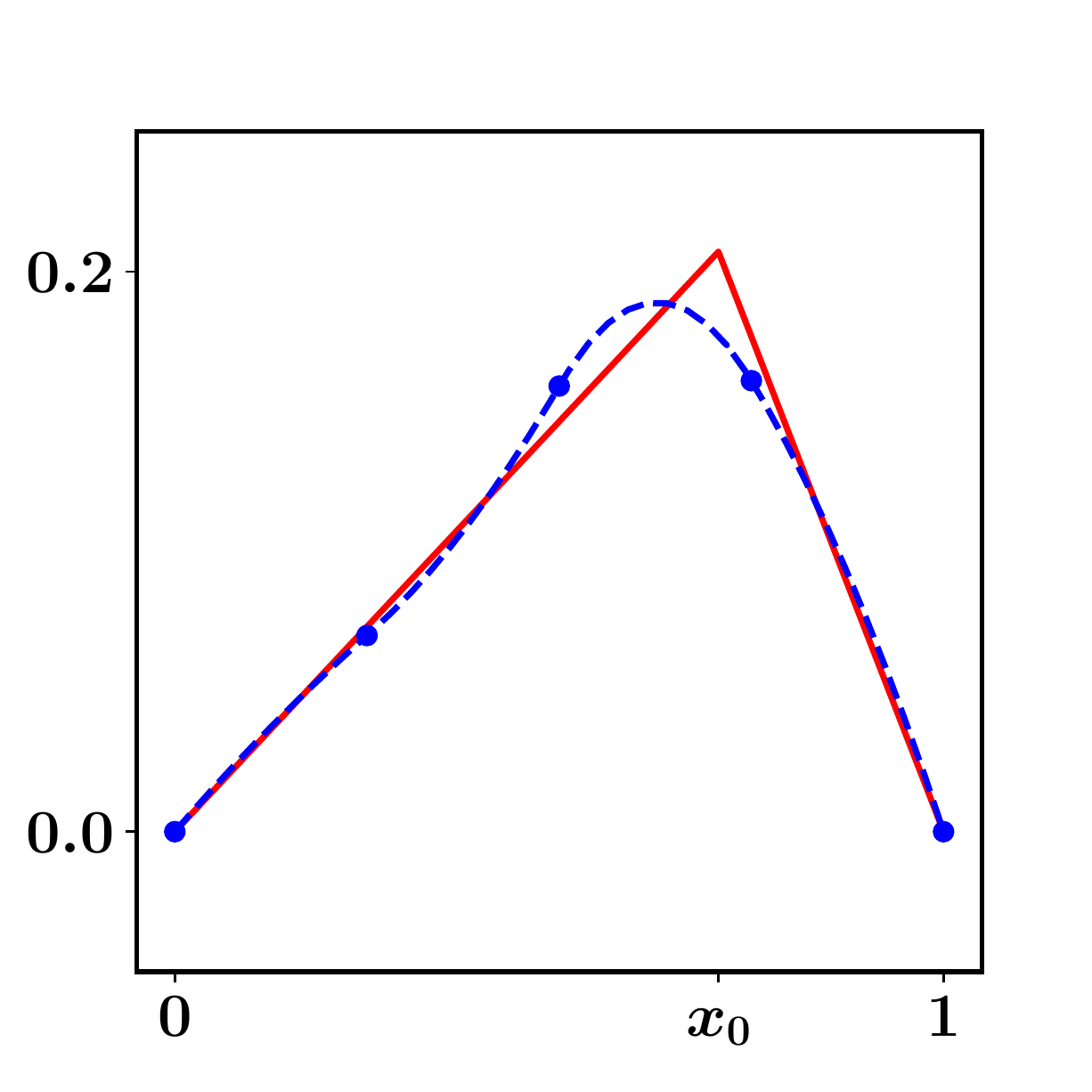}
\includegraphics[width=0.32\textwidth]{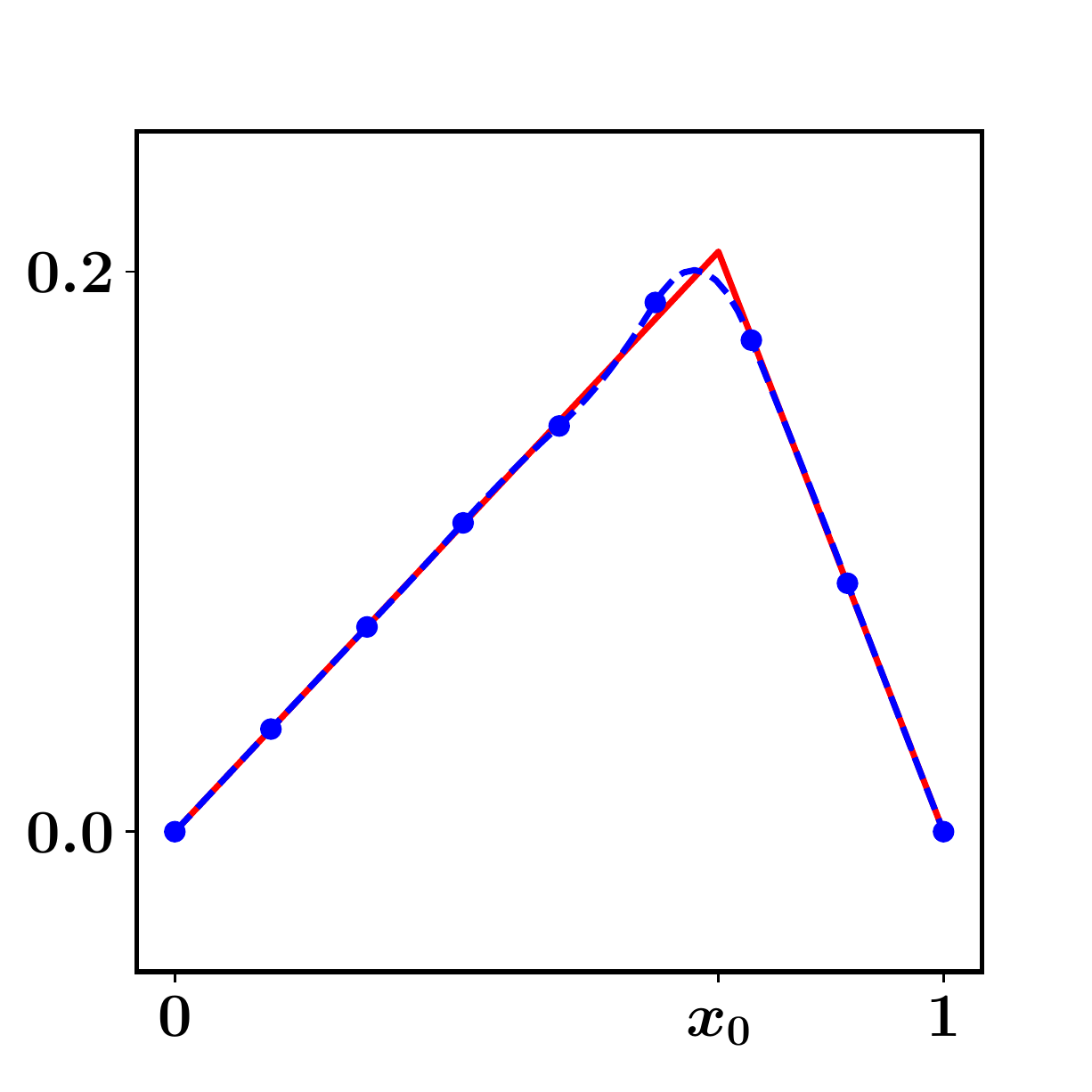}
\includegraphics[width=0.32\textwidth]{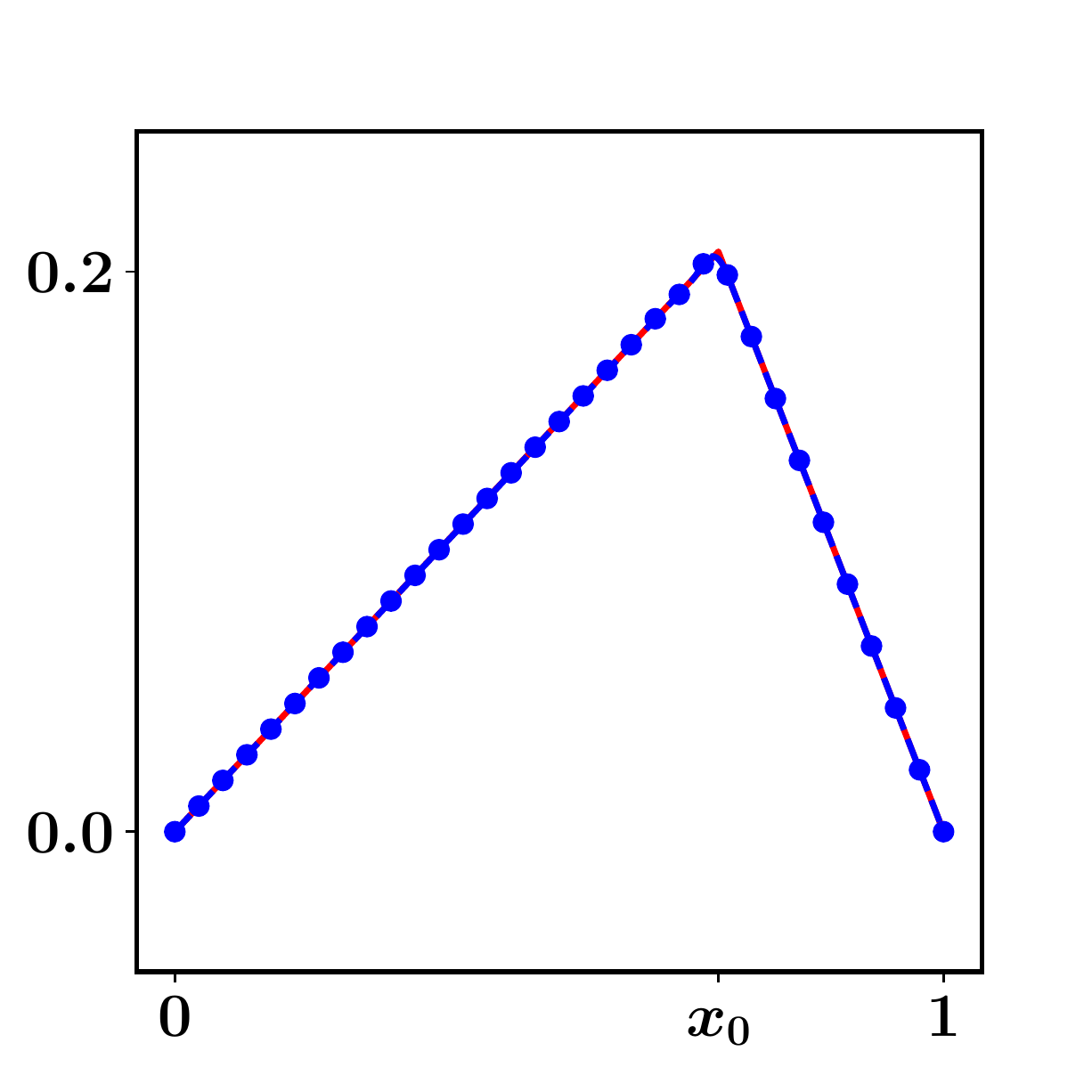}
\caption{Regularized solution $u_n$ of problem~\eqref{ODE:REG} (dashed line), compared with the exact solution~\eqref{exact:1} (continuous line), with Dirac delta supported on $x_0=\sqrt{2}/2$ and $a=0$. The inputs for $u_n$ have been obtained from projected Dirac deltas onto uniform meshes of $n=4, 8$ \& $32$ elements, respectively.}
\label{fig:ODE_uniform}
\end{figure} 
In Figure~\ref{fig:ODE_uniform}, we show the solutions obtained by approximating equation~\eqref{ODE:REG} using a Galerkin squeme with conforming $\mathbb P_2(\mathcal T_h)$ Lagrange finite elements over the same mesh that defines $\delta_n\in G_n$, which in this case corresponds to uniform meshes of $n=4, 8$ \& $32$ elements. We have chosen $x_0=\sqrt{2}/2$ to make sure that $x_0$ never coincides with a node of the meshes. Moreover, we have considered $a=0$, which implies that the exact solution of~\eqref{ODE:REG} is indeed contained in the discrete space $\mathbb P_2(\mathcal T_h)$, and thus, there is no discretization error. The convergence rates of the error $\|u-u_n\|_{H_0^1}$ coincide with the convergence rates of $\|\delta_{x_0}-\delta_n\|_{H^{-1}}$ (or, more precisely, with $\|r_m\|_V$), as can be observed from Figure~\ref{fig:rates1dLaplacian}.
\begin{figure}[htb]
\centering
\includegraphics[width=0.6\textwidth]{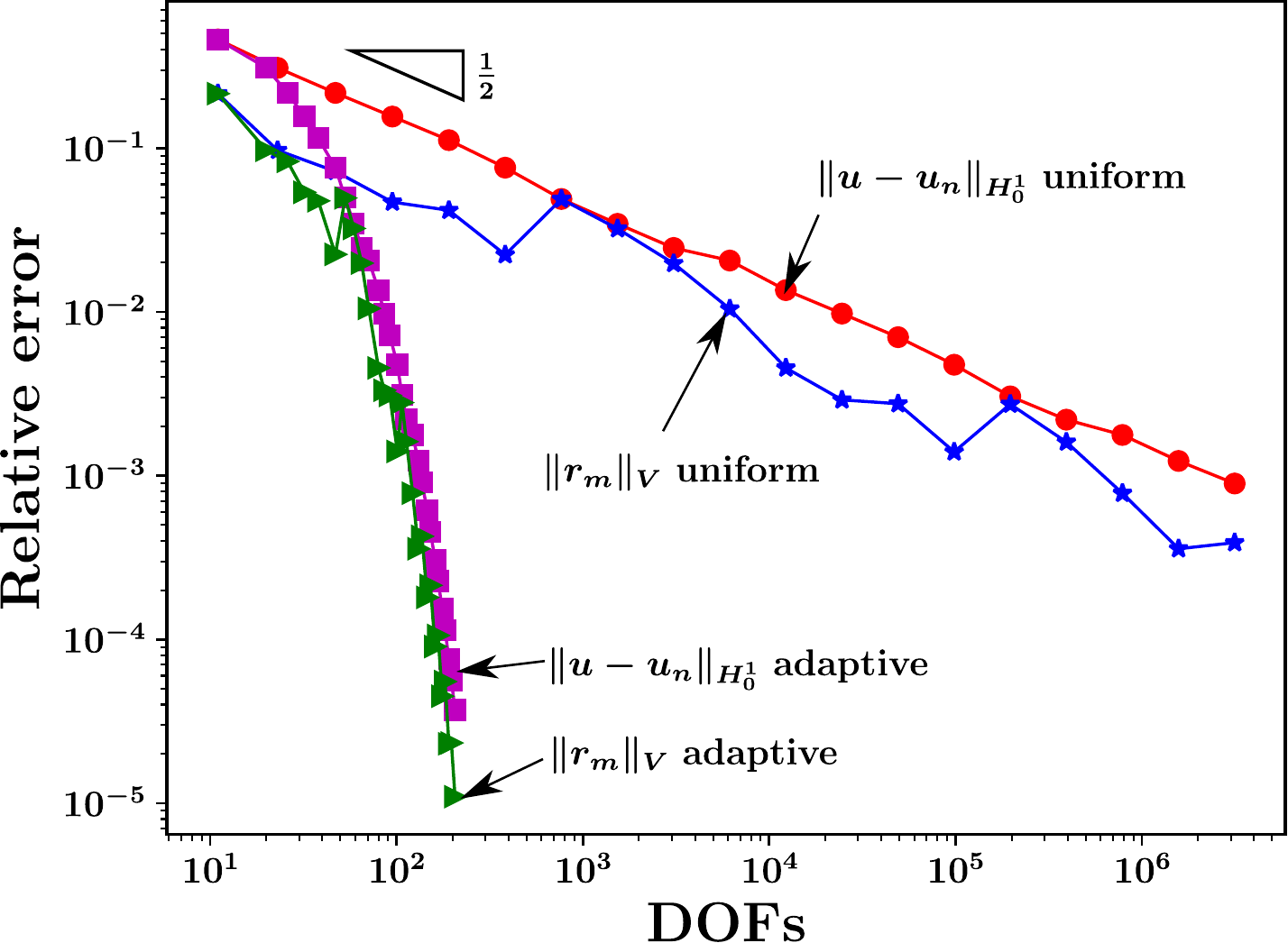}  
\caption{Convergence rates for the solution of the regularized problem~\eqref{ODE:REG}. 
}
\label{fig:rates1dLaplacian}
\end{figure}

The next experiment intends to answer the question: {\it can we approximate the solution of problem~\eqref{ODE} up to a given desired precision?} 
Inspired in the adaptive algorithm proposed by Bonito et al.~\cite{NOCHETTO2013}, we propose a two-step adaptive procedure. The first step controls the {\it regularization error}, while the second step controls the {\it discretization error}. To fix ideas, let $u_{n,h}$ be the discrete approximation of problem~\eqref{ODE:REG} obtained by an adaptive procedure. Observe that:
$$
\|u - u_{n,h}\|_{H_0^1}\leq \|u - u_{n}\|_{H_0^1}+\|u_n - u_{n,h}\|_{H_0^1}\leq 
{1\over\gamma} \|\delta_{x_0}-\delta_n\|_{H^{-1}} +\|u_n - u_{n,h}\|_{H_0^1},
$$
where $\gamma>0$ is the stability constant of our differential operator ($\gamma=1$ in this particular example). In the first step, we perform an adaptive projection of the Dirac delta until reaching some prescribed tolerance. This will control the regularization error $\|\delta_{x_0}-\delta_n\|_{H^{-1}}$ and will deliver a projected delta $\delta_n$ together with an adapted mesh $\mathcal T_h=\{T_i\}_{i=1}^n$. In the second step, we use the source $\delta_n$ obtained in the first step and solve problem~\eqref{ODE:REG} adaptively considering $\mathcal T_h$ as the initial mesh, until reaching the prescribed tolerance. This will control de discretization error $\|u_n - u_{n,h}\|_{H_0^1}$. The general procedure  is depicted in Algorithm~\ref{algorithm}. In particular, for the second step we have used a standard local \emph{a posteriori} error estimator $ \eta_ {T_i}$, similar to the one used in~\cite{NOCHETTO2013}. Since the RHS $\delta_n$ is a piecewise constant function, there is no data oscillation in this case (see, e.g.~\cite{MorNocSie_SINUM2000}).
\begin{algorithm}
\begin{algorithmic}[1]
\State Global $\mbox{tol}>0$, $\alpha\in(0,1)$
\Procedure{Source regularization algorithm}{}
\State $\textit{Input} \gets \text{RHS}:=\delta_{x_0}, \text{mesh}:= \mathcal T_h=\{T_i\}_{i=1}^n$, %
$q > 1$%
\State $(r_m,\delta_n)\gets \text{Solve scheme } \eqref{fully1} $	
 	\While{$ \text{tol}<\|r_m\|^{q-1}_{V}$} 
		\If{$\|r_m\|^{q-1}_{V(T_i)}>\alpha \max \{\|r_m\|^{q-1}_{V(T_i)}\} $}
		 \State Refine the element $T_i$	
		\EndIf
		\State Update \text{mesh} $\mathcal{T}_h$ and go to step 3
\EndWhile
\Return  $\mathcal{T}_h$, $\delta_n$

\EndProcedure
\Procedure{Approximating the solution of the regularized PDE}{}	
	\State \textit{Input} $\gets \text{RHS}:=\delta_n, \text{mesh}:=\mathcal T_h$
	\State  $u_{n,h} \gets$ Solve the problem \eqref{ODE:REG} by Galerkin method
	\State Compute local \emph{a posteriori} estimators $\eta_{T_i}$
\While{$ \mbox{tol} \leq \sqrt{\sum\eta_{T_i}^2}$ } 	
	\If{$\eta_{T_i}> \alpha \max\{ \eta_{T_i}\}$}
	\State Refine the element $T_i$			
	\EndIf
	\State Update \text{mesh} $\mathcal{T}_h$ and go to step 10
\EndWhile 	
	\Return  $\mathcal T_h$, $u_{n,h}$.		
\EndProcedure
\end{algorithmic}
\caption{Approximating the solution of a PDE by regularization}
\label{algorithm}
\end{algorithm}  
Figure~\ref{fig:ExactVSApprox} (left) depicts the error in $H_0^1$ semi-norm of the adapted discrete solution v/s the chosen tolerance in Algorithm~\ref{algorithm}. 
We report here that smaller tolerances would lead to huge condition numbers in the second step of Algorithm~\ref{algorithm}, making results unreliable.
In Figure~\ref{fig:ExactVSApprox} (right), the final discrete solution, computed using Algorithm~\ref{algorithm} with $\mbox{tol}=0.001$ and $\alpha =0.5$, is compared with the exact solution.
We have set the values $x_0=\sqrt{2}/2$ and $a=2$ in~\eqref{ODE}.
\begin{figure}[htb]
\centering
\includegraphics[width=0.44\textwidth]{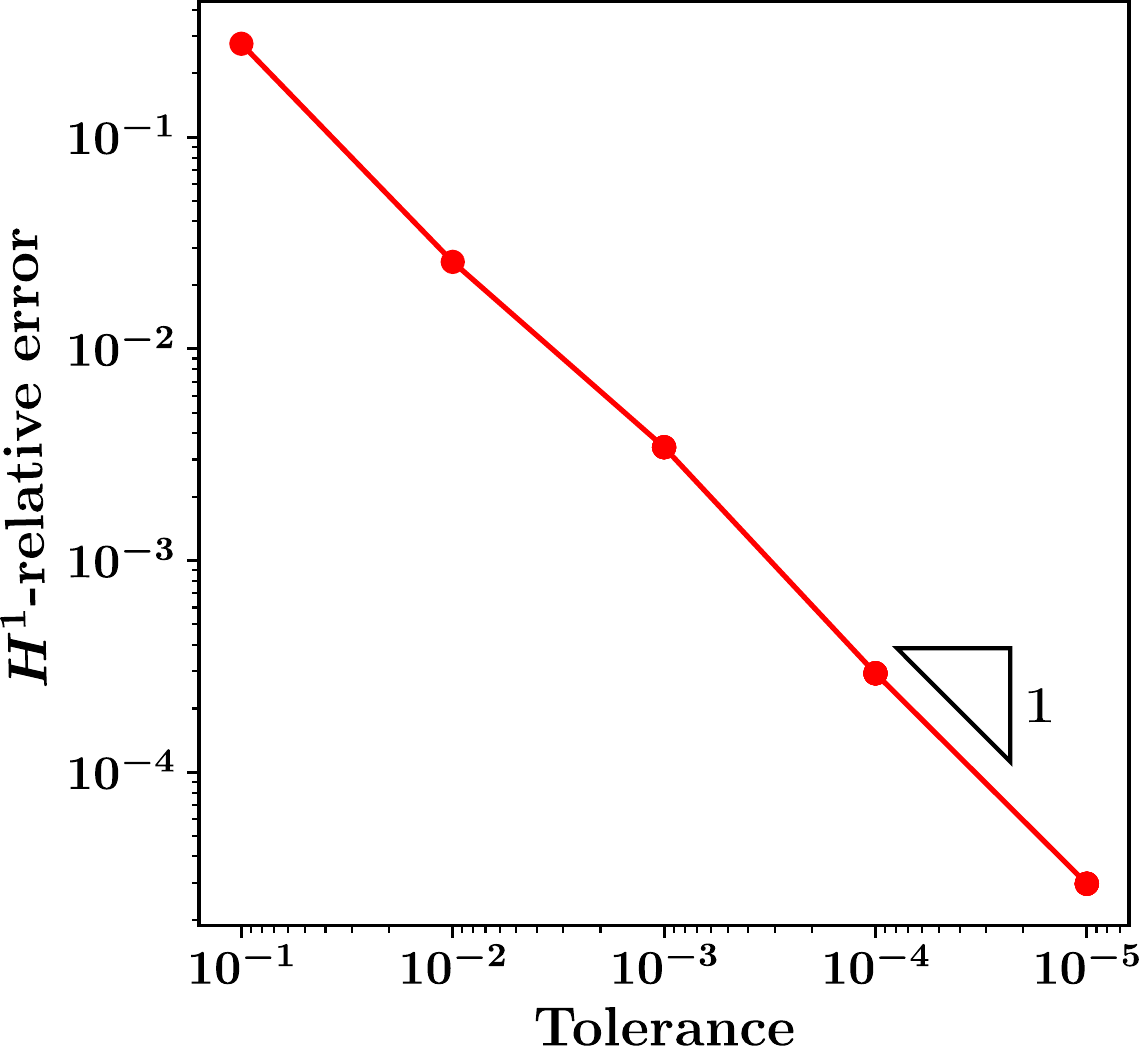}
\hspace{0.04\textwidth} 
\includegraphics[width=0.42\textwidth]{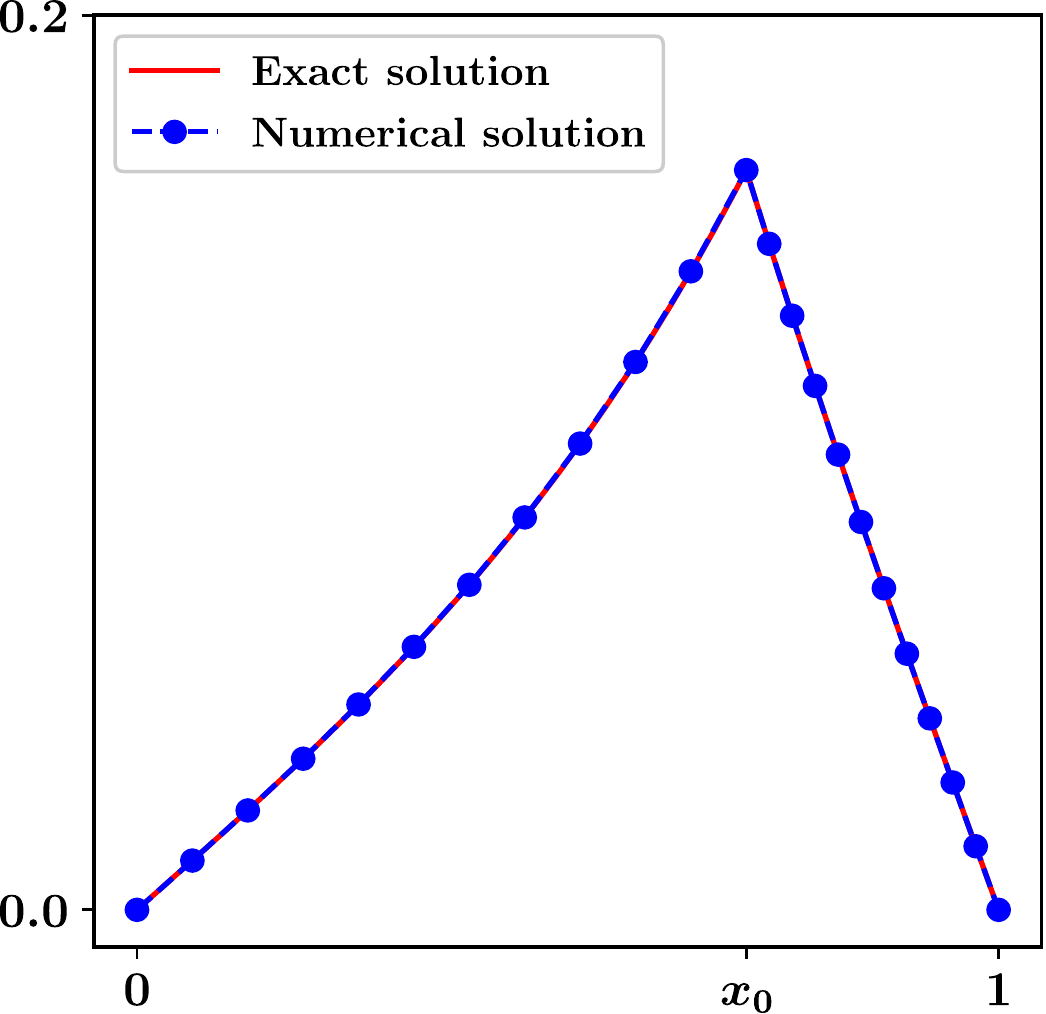}  
\caption{Left: Relative $\|u-u_{n,h}\|_{H_0^1}$ error v/s tolerance of Algorithm~\ref{algorithm}. Right: Exact solution of~\eqref{ODE} and its adapted numerical approximation using Algorithm~\ref{algorithm}.}
\label{fig:ExactVSApprox}
\end{figure}

\subsubsection{Two dimensional  Dirac's Delta projection.}
Let $\Omega\subset \mathbb{R}^2$ be the unitary circle centered at $x_0=(0,0)$ and  
 consider the Poisson problem: 
\begin{equation}\label{eq:Poisson2D}
\left\{ \begin{array}{rll}
-\Delta u &=\delta_{x_0} \qquad &\mbox{ in } \Omega, \\
u&=0 \qquad&\mbox{  on } \partial \Omega,
\end{array}\right.
\end{equation}
along with its exact solution:
\begin{equation}\label{exact2D}
u(x)= -\frac{1}{2\pi}\ln|x|.
\end{equation}
We are going to project the Dirac delta into a piecewise constant space $G_n$ and then proceed to approximate PDE~\eqref{eq:Poisson2D} using this projected input.
So let  $\{\mathcal{T}_h\}=\{T_i\}_{i=1}^n$ be a simplicial mesh and let $G_n$ be defined by~\eqref{trialspace:1}. In this example we take $V_m:=\mathbb P_1(\mathcal T_h)\cap W_0^{1,q}(\Omega) + \mathbb B_n(\mathcal T_h)$ (see expressions~\eqref{Bubble}~and~\eqref{eq:P1}). 
In Figure~\ref{error_delta_2d} we show convergence rates of the projected Dirac delta using uniform and adaptive $h$-refinements, for diverse choices of $p\in(1,2)$. As in the one dimensional case, we observe that  the convergence rates of uniform $h$-refinements are related with the difference between the regularity exponent $s = 1$  and the critical regularity exponent $s^* = 2/q$ (recall that in 2D we have $W_0^{s,q}(\Omega)\subset\mathcal C(\Omega)$ whenever $sq>2$).
The graphical representation is constructed in terms of the square root of the degrees of freedom (DOFs$^{1/2}$). 
 In the case of adaptive $h$-refinements, for each value of $p$, we present the first 15 iterations of the adaptive algorithm
using the marking criteria $\alpha=0.5$ (see Algorithm~\ref{algorithm}). Again, exponential convergence is observed due to the fact that the source is localized in only one point.
\begin{figure}[htb]
\centering
\includegraphics[width=0.7\textwidth]{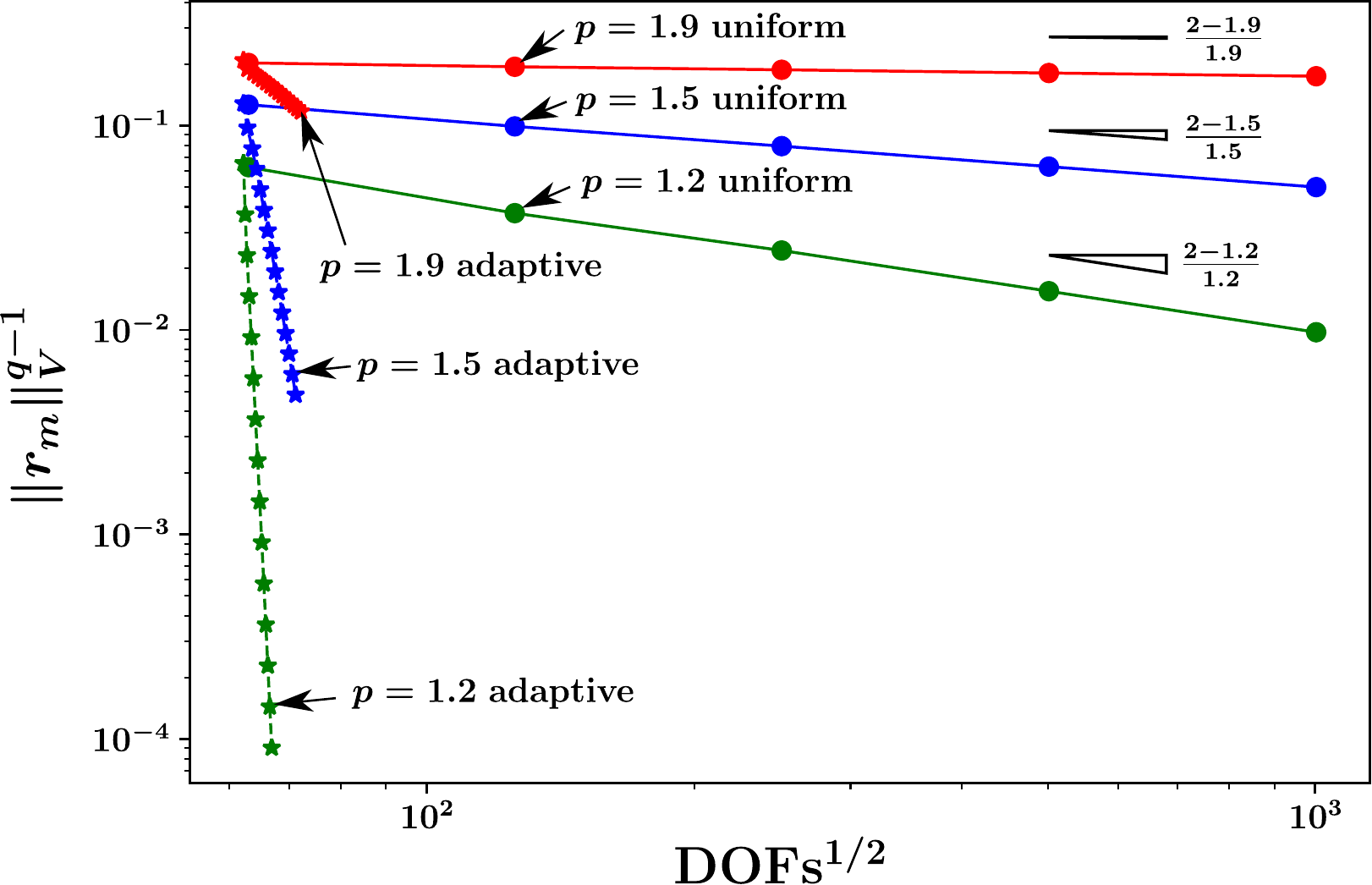}  
\caption{2D Dirac delta: convergence rates of uniform and adaptive $h$-refinements for several values of $p$.}
\label{error_delta_2d}
\end{figure}

In Figure~\ref{fig:Green2D}, we can appreciate how numerical approximations of PDE~\eqref{eq:Poisson2D} are approaching the singular exact solution~\eqref{exact2D} when using regularized sources obtained from different steps of the adaptive $h$-refinements procedure. These results were obtained employing $p=1.9$ and marking criteria parameter $\alpha=0.5$ in the adaptive procedure.
For each adapted mesh $\mathcal T_h$, the PDE has been approached by means of a conforming $\mathbb P_1(\mathcal T_h)$ Galerkin squeme.
\begin{figure}[htb]
\centering
\includegraphics[width=0.31\textwidth]{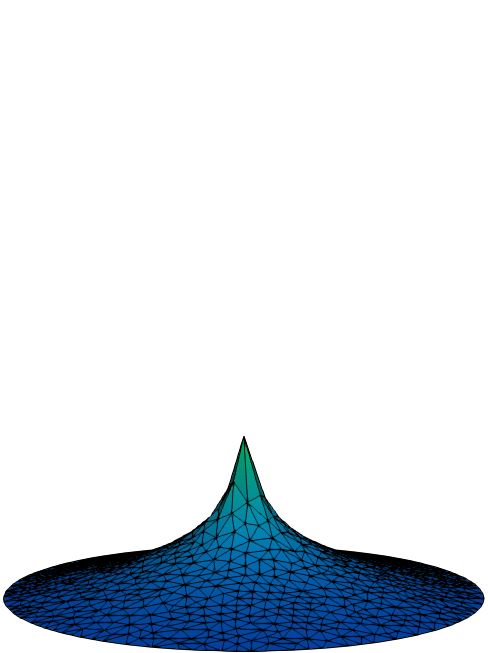} \hspace{0.01\textwidth}
\includegraphics[width=0.31\textwidth]{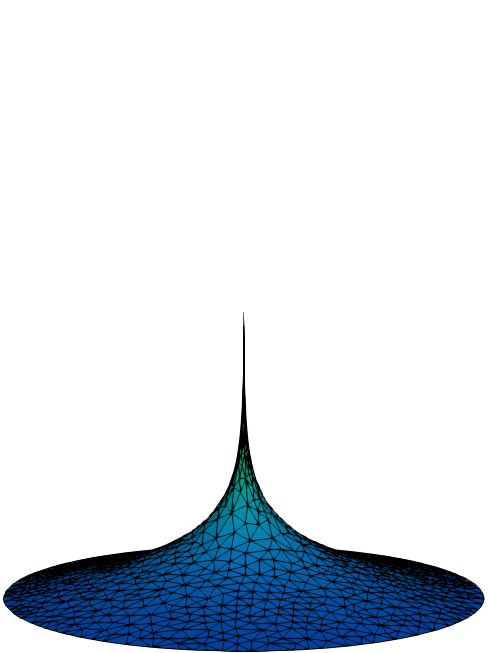}\hspace{0.01\textwidth}
\includegraphics[width=0.31\textwidth]{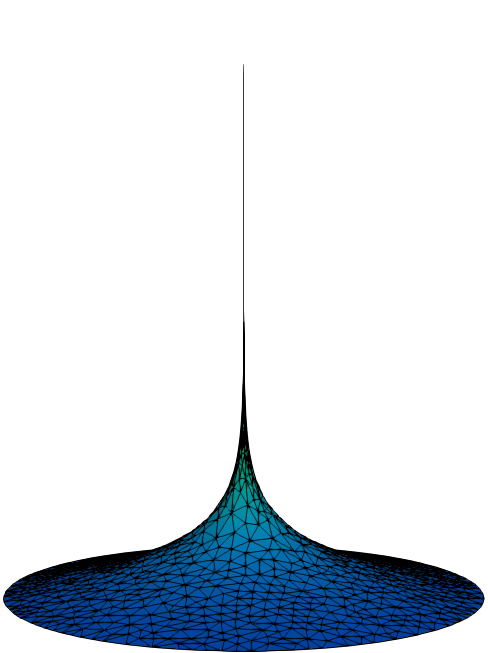}
\caption{Numerical approximations of the PDE~\eqref{eq:Poisson2D} using adaptive regularized sources.}
\label{fig:Green2D}
\end{figure}

\subsection{Line Source}\label{sec:line_source}
The last experiment is inspired by Example~\ref{ex:line_sources}. Let  $\Omega:=(0,1)^2\subset \mathbb{R}^2$ be the unit square, and let $\Gamma\in\overline{\Omega}$ be the segment:
$$\Gamma:=\left\{ \left( t,(t-0.5)^2+0.5\right)\in \mathbb{R}^2\,:\, 0.15\leq t\leq 0.85  \right\}.$$
We are going to project the linear functional $\ell\in W^{1,q}(\Omega)$ defined by:
$$
\ell(v)=\int_\Gamma v\,ds.$$ 
More generally, $\ell$ is well-defined, linear and continuous, for any $v\in W^{s,q}(\Omega)$, whenever $sq\geq 1$. Indeed, $\left.v\right|_{\Gamma}\in W^{s-{1\over q},q}(\Gamma)$.
Let $\mathcal{T}_h=\{T_I\}_{i=1}^{n}$ be a simplicial mesh (not necessarily aligned with $\Gamma$), and let us  consider the discrete spaces $G_n$ (defined in~\eqref{trialspace:1}) and $V_m:=\mathbb P_1(\mathcal T_h)\cap W_0^{1,q}(\Omega) + \mathbb B_n(\mathcal T_h)$ (see expressions~\eqref{Bubble}~and~\eqref{eq:P1}). 
\begin{figure}
\centering
\includegraphics[width=0.7\textwidth]{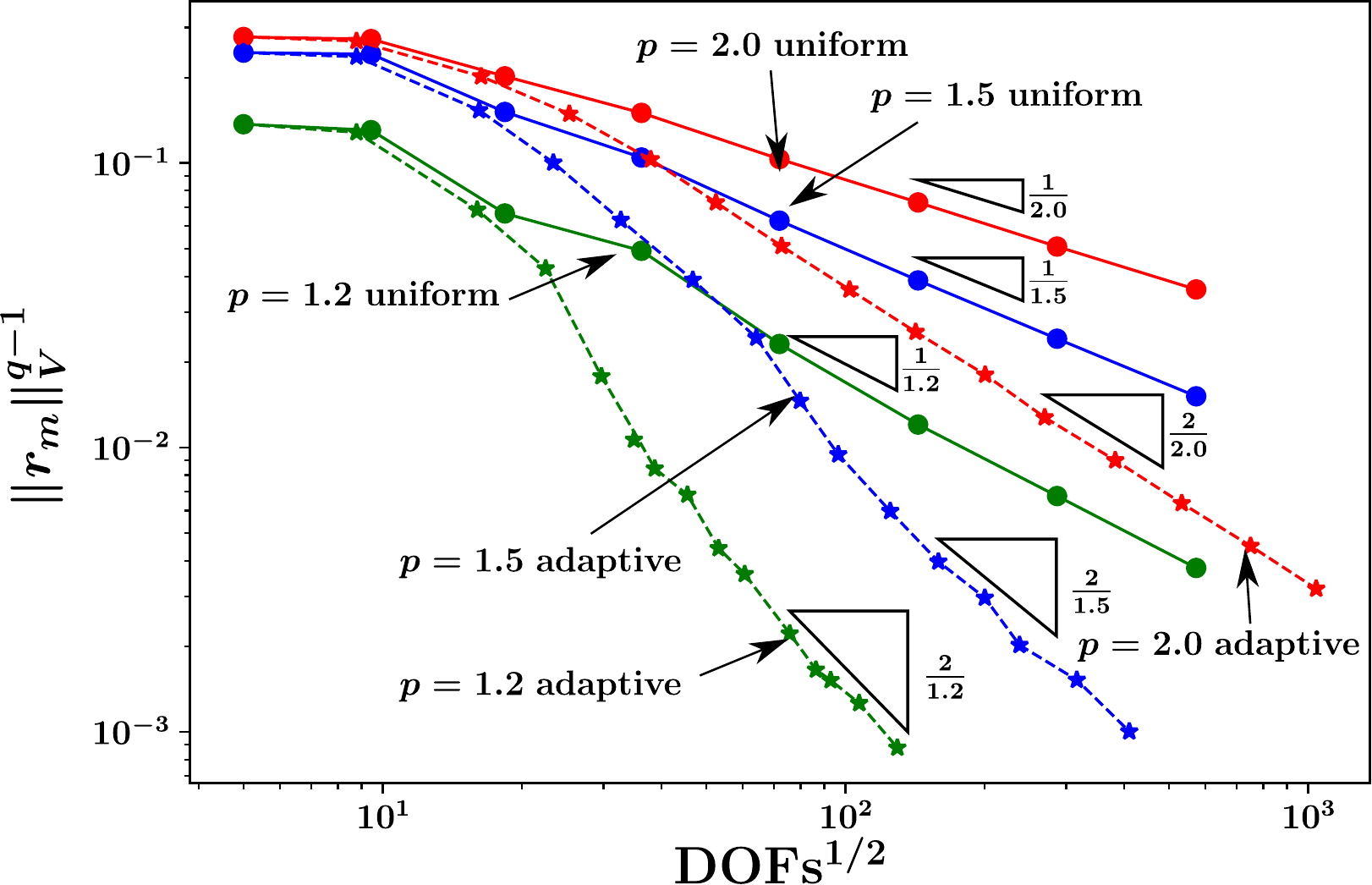}
\caption{Line source: convergence rates of uniform and adaptive $h$-refinements for several values of $p$.}%
\label{fig:crack_error}
\end{figure}
In Figure~\ref{fig:crack_error}, we show convergence rates of the projection of $\ell$ using uniform and adaptive $h$-refinements, for several values of $p$. We observe that the convergence rates of uniform $h$-refinements are close to the difference between the current regularity exponent $s = 1$ and the critical regularity exponent $s^* = 1/q$ (cf.~Remark~\ref{rem:rates}). This graphical representation is constructed in terms of the square root of the degrees of freedom (DOFs$^{1/2}$). 
We have observed that adaptive $h$-refinements practically double the rates of uniform $h$-refinements.
%
\begin{figure}
\centering
\includegraphics[width=0.3\textwidth]{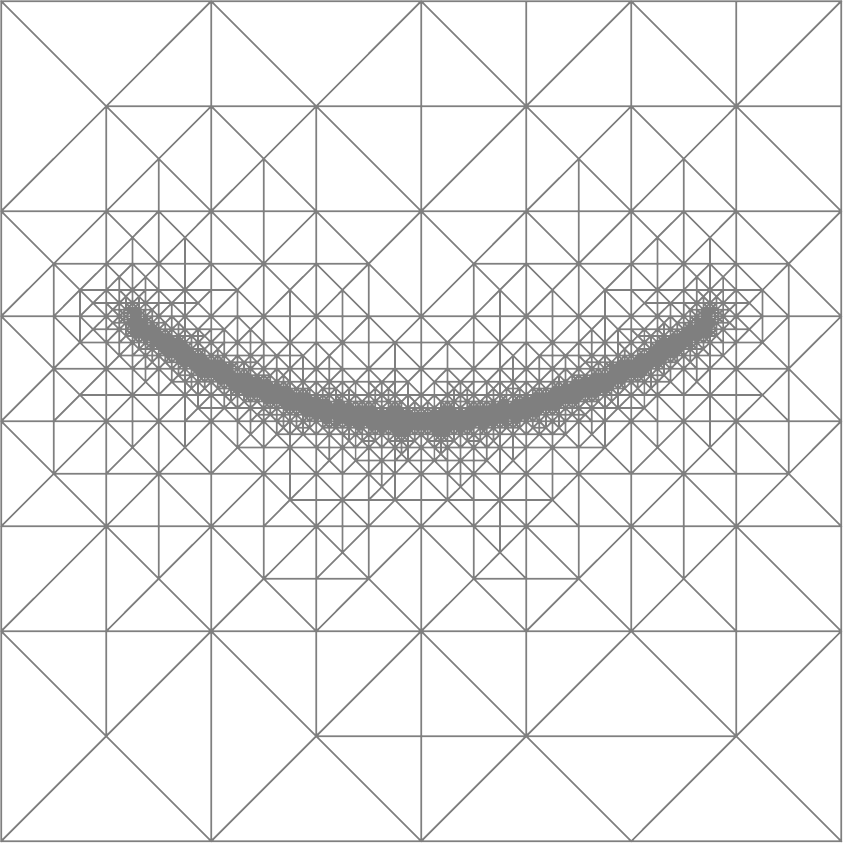} \hspace{0.01\textwidth}
\includegraphics[width=0.3\textwidth]{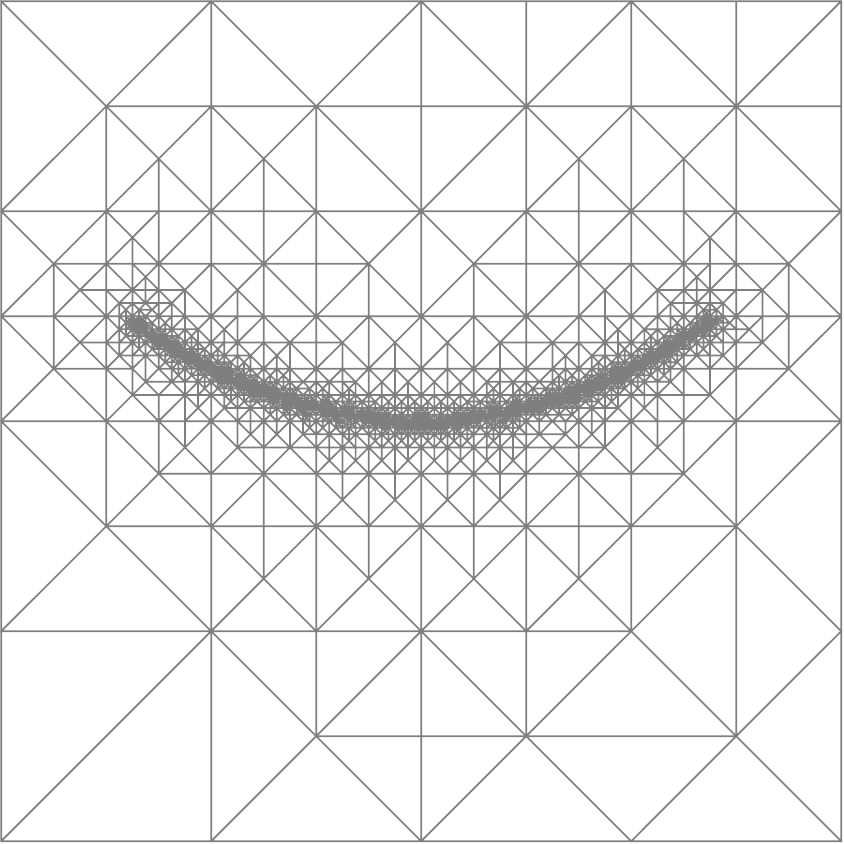} \hspace{0.01\textwidth}
\includegraphics[width=0.3\textwidth]{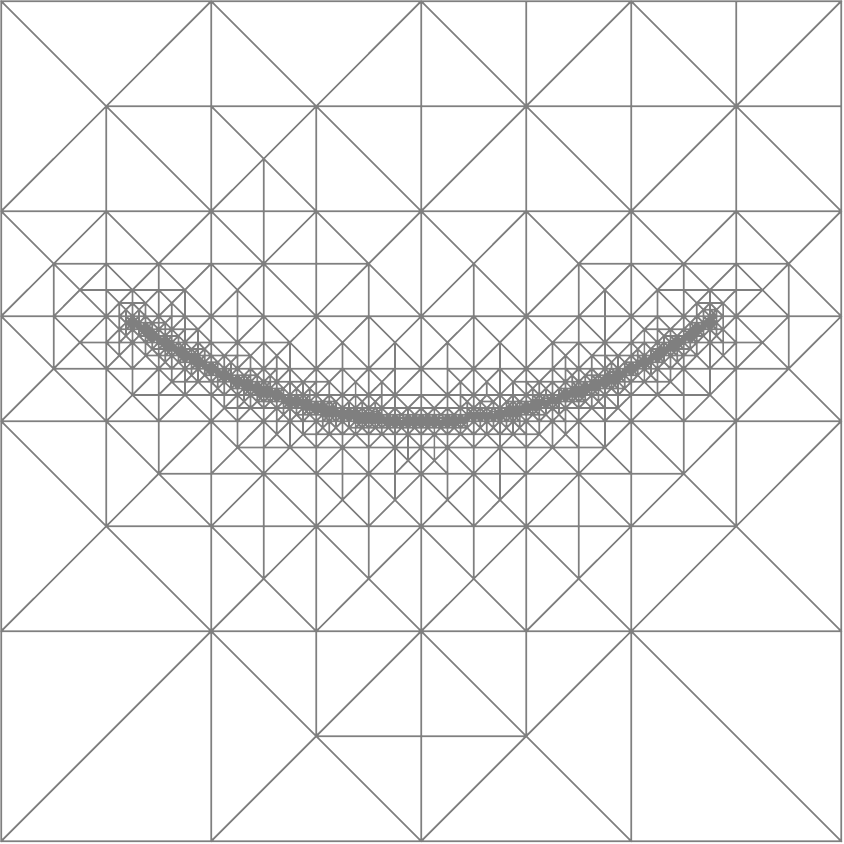}
\caption{Line source: Adaptive mesh at a comparable NDOFs, for $p=2,1.5$ \& $1.2$,  respectively.}
\label{fig:mesh_line}
\end{figure}
On another hand, Figure~\ref{fig:mesh_line} shows a sequence of adaptive meshes, obtained for $p=2, 1.5$ and $1.2$, at a comparable number of degrees of freedom (NDOFs), using the marking criteria parameter $\alpha=0.4$. As expected, refinements are concentrated along the support of the distribution $\ell$, i,e., $\Gamma$. We observe better localization of refinements as the value of $p$ decreases. 

\section{Conclusions}\label{sec:conclusions}

Based on the recent theory of residual minimization in Banach spaces developed in~\cite{MugVdZ_SINUM2020}, we proposed in this work a method to regularize rough linear functionals, projecting them into piecewise polynomial spaces. The projections has been performed in terms of discrete-dual Banach norms. 
Particularly, we have studied functionals involving actions over test functions with a certain regularity, i.e., functionals in negative Sobolev spaces. 
Our approach has two remarkable advantages. First, the regularization can be obtained within low-order piecewise polynomial spaces. Therefore, if such a regularization is used on the right-hand-side of a finite element system, then exact numerical integration can be implemented via Gaussian quadrature formulae. Second, as every residual minimization approach does, the method computes a built-in residual representative, which is proven to be a reliable and efficient \textit{a posteriori} error estimator. Indeed, we have used this estimators to drive adaptive procedures delivering regularized functionals up to any desired precision in the underlaying discrete-dual norm.
We have observed superior performance of adaptive $h$-refinements in terms of convergence rates.

\par

On another hand, in terms of the discrete stability of our method, we exhibit 
two compatible trial-test discrete pairing that can be used in every problem involving rough functionals acting on $W_0^{1,q}(\Omega)$ (see Propositions~\ref{Fortin_space} and~\ref{prop:compatible2}).

Some future research challenges may include the regularization of rougher functionals, such as dipole sources in electroestatics~\cite{AloCamRodVal_CAMWA2014} or derivatives of Dirac deltas, which appear naturally in high-order PDEs modeling of elastic plates and beams~\cite{YavSarMoy_IJSS2000}.

\appendix
\section{Proof of Proposition 9}%
\normalsize
\label{appendix}
Through this proof, the symbol $\lesssim$ will denote \emph{less or equal} up to a mesh-independent constant.
Recall that we are under the hypothesis of shape-regular simplicial meshes $\mathcal T_h=\{T_i\}_{i=1}^n$ and $q>d$. 
Let $\Pi_1:V\to \mathbb P_1(\mathcal T_h)$ be the \emph{Lagrange interpolant} operator, such that for any $T_i\in\mathcal T_h$, the restriction 
$\Pi_1(v)|_{T_i}\in \mathbb{P}_1(T_i)$ satisfies the local  estimation
\begin{equation}\label{local::intepor}
\|\Pi_1 v -v \|_{L^q(T_i)}\lesssim  h_i \|\nabla v \|_{L^q(T_i)},\qquad\forall v\in V,
\end{equation}
where $h_i=\mbox{diam}(T_i)$ (see, e.g.,~\cite[Theorem~1.103]{ern2013theory}). Moreover, $\Pi_1$ is uniformly bounded with respect to mesh parameters (see, e.g.,~\cite[Corollary~1.109]{ern2013theory}), i.e.,  
\begin{equation}\label{bound::interpolant}
\| \Pi_1 v \|_{V}\lesssim\|v \|_{V},\qquad\forall v\in V.%
\end{equation}
For each simplex $T_i\in\mathcal{T}_h$ let $b_i:T_i\to\mathbb R$ be the bubble function defined in~\eqref{eq:bubble}. Consider now the (Fortin) operator $\Pi: V\mapsto V_m$ locally defined by:
\begin{equation}\label{eq:Pi}\notag
\Pi v|_{T_i}:= \left.\Pi_1 v\right|_{T_i} + \alpha_i b_i,
\end{equation}
where $\alpha_i\in\mathbb R$ is chosen such that the following equality holds true:
\begin{equation}\label{Fortin::eq1}
\int_{T_i} \Pi v=\int_{T_i} v,\qquad \forall v\in V.
\end{equation}
Solving for $\alpha_i$ we get:
\begin{equation}\label{Fortin::eq2}
\alpha_i= \left(\int_{T_i} b_i \right)^{-1} \int_{T_i}(v-\Pi_1 v).
\end{equation}

We will proceed to prove that  $\Pi $ satisfies the  Fortin conditions~\eqref{constC} and~\eqref{Fortin:c3}. First, for any $g_n=\sum_{i=1}^n\beta_i\mathcal G_i\in G_n$, and any $v\in V$, we have:
\begin{align*}
\langle g_n, \Pi v\rangle _{V^*,V}&=\sum_{i=1}^n \beta_i \int_{T_i} \Pi v
=\sum_{i=1}^n \beta_i  \int_{T_i} v=\langle g_n,v\rangle_{V^*,V}
\tag{by~\eqref{trialspace:1} and~\eqref{Fortin::eq1}}.%
\end{align*}
Hence, condition \eqref{Fortin:c3} is satisfied.  To prove~\eqref{constC}, 
let us consider the reference element $\hat{T}$  and the affine mapping $F_i: \hat{T}\to T_i$ such that $F_i(\hat{x})=A_i\hat{x}+y_i $. Thus, if $b_{\hat{T}}$ denotes the bubble in $\hat{T}$, then:
$$
b_i(x)=b_{\hat{T}}\circ F_i^{-1}(x), \qquad \forall x\in T_i.
$$
Moreover, we have the following estimations (see~\cite[Lemma 1.100]{ern2013theory}):
$$
|\det A_i |= \frac{|T_i|}{|\hat{T}|}\qquad\mbox{ and }\qquad \| A_i^{-1}\|\leq \frac{h_{\hat{T}}}{\rho_i}
$$
where $\rho_i$ is the radius of the largest ball inscribed in $T_i$ and $\|\cdot\|$ is the matrix norm subordinated to the Euclidean norm in $\mathbb R^d$. Notice that 
$h_{\hat{T}}$ and $|\hat{T}|$ do not depend on the mesh.  
Using the change-of-variables theorem we get:
$$
\int_{T_i}b_i=  \int_{\hat{T}}b_{\hat{T}}|\det A_i|=\frac{|T_i|}{|\hat{T} |}\int_{\hat{T}}b_{\hat{T}} = C |T_i|,
$$
where $C>0$ is a mesh independent constant.
Hence, we can estimate $|\alpha_i|$ (see~\eqref{Fortin::eq2}) as follows: 
\begin{align}\label{alpha::estimate}
|\alpha_i |\lesssim 
\frac{1}{|T_i|} \| v-\Pi_1 v\|_{L^q(T_i)} | T_i|^{1\over p}%
= {1\over |T_i|^{1\over q}}\| v-\Pi_1 v\|_{L^q(T_i)}
\lesssim  {h_i\over |T_i|^{1\over q}}\| \nabla v\|_{L^q(T_i)},
\end{align}
where we have used~\eqref{local::intepor} in the last inequality.
Additionally, we estimate the bubble function semi-norm by the following classical result (see~\cite[Lemma 1.101]{ern2013theory}):
\begin{equation}\label{bubble::estima1}
\|\nabla b_i \|_{L^{q}(T_i)}\lesssim   \|A_i^{-1}\| |\det A_i |^{1\over q}\|\nabla b_{\hat{T}} \|_{L^q(\hat{T})}
\lesssim {|T_i|^{1\over q}\over \rho_i}.
\end{equation}
Combining~\eqref{alpha::estimate} and~\eqref{bubble::estima1} we get:
\begin{equation}\label{eq:alphab}
|\alpha_i |\| \nabla  \mathit{b}_{T_i}\|_{L^q(T_i)}\lesssim {h_i\over \rho_i}\| \nabla v\|_{L^q(T_i)} \lesssim\| \nabla v\|_{L^q(T_i)},
\end{equation}
where the last inequality holds true because of the shape-regularity of the mesh. 
Finally,~\eqref{constC} holds true since
\begin{align*}
\|\nabla  \Pi v\|^q_{L^q(\Omega)} &= \sum_{i=1}^n\|\nabla\Pi v \|^q_{L^q(T_i)}\leq 2^{q-1}\sum_{i=1}^n \left(  \|\nabla \Pi_1 v \|^q_{L^q(T_i)} +|\alpha_i|^q \| \nabla b_i \|^q_{L^q(T_i)}\right)\\%
&\lesssim  \| \nabla \Pi_1 v\|^q_{L^q(\Omega)}+ \sum_{i=1}^n \| \nabla v \|^q_{L^q(T_i)}  \lesssim  \|\nabla v\|^q_{L^q(\Omega)},
\end{align*}
where we have used H\"older inequality, together with~\eqref{bound::interpolant} and~\eqref{eq:alphab}.

\section{Proof of Proposition 10}%
\label{appendix2}
Again, through this proof, the symbol $\lesssim$ will denote \emph{less or equal} up to a mesh-independent constant. The proof of Proposition~\ref{prop:compatible2} requires the following previous lemma.
\begin{lemma}\label{lem:orthogonality}
Under the hypothesis of Assumption~\ref{ass:mesh}, 
let $\{\varphi_i\}_{i=1}^{N_v}$ be the set of nodal basis functions spanning $\mathbb P_1(\mathcal T_h)\cap W_0^{1,p}(\Omega)$. 
There exists a (bi-orthogonal) set $\{\psi_i\}_{i=1}^{N_v}\subset\mathbb P_2(\mathcal T_h)\cap W_0^{1,q}(\Omega)$ such that:
\begin{equation}\label{eq:orthogonality}
\int_\Omega \varphi_i\psi_j = \eta_i\,\delta_{ij}, \quad\forall i,j=1,...,N_v,
\end{equation}
where $\delta_{ij}$ denotes the Kronecker delta, and $\eta_i$ denotes scaling constant of the patch $P_i=\overline{\operatorname{supp}\varphi_i}$ (see Assumption~\ref{ass:mesh}). 
Moreover we have the estimate:
\begin{equation}\label{eq:psi_estimate}
\|\nabla\psi_i\|_{L^q(\Omega)}\lesssim {\eta_i^{1\over q}\over \displaystyle\min_{T\in P_i}\rho_T}, \quad \forall i=1,...,N_v,
\end{equation}
where $\rho_T$ denotes the diameter of the largest ball that can be inscribed in $T$.
\end{lemma}
\begin{proof}
By Assumption~\ref{ass:mesh}, the patch $P_i$ supporting the nodal basis function $\varphi_i$  has reference patch $\widehat P_i$, having the same configuration of elements of $P_i$. 
Hence, the patch $\widehat P_i$ has a single interior vertex denoted by $x_0$. Each other vertex of this patch is linked to $x_0$ through a unique interior edge. Let $n_e$ be the number of exterior vertices of the patch $\widehat P_i$ (equivalently, the set of interior edges of the patch $\widehat P_i$). The local $\mathbb P_1$ trial basis functions for this reference patch will consist in $n_e+1$ shape functions $\{\widehat\varphi_j\}_{j=0}^{n_e}$, where $\widehat\varphi_0$ denotes the trial function associated with the interior vertex $x_0$. We want to construct a $\mathbb P_2$ test function $\widehat\psi_i$, supported on the patch $\widehat P_i$, such that:
\begin{equation}\label{eq:delta_ij}
\int_{\widehat P_i}\widehat\varphi_j\widehat\psi_i=\delta_{0j}, \quad\forall j=0,1,..,n_e.
\end{equation}
We provide a two-dimensional procedure to construct such a $\widehat\psi_i$, which can be easily extended to three dimensions. Let us denote by $\{\widehat T_1,...,\widehat T_{n_e}\}$ the set of simplicial elements that make up the patch $\widehat P_i$, and let $\{x_1,...,x_{n_e}\}$ be the set of exterior vertices of the patch $\widehat P_i$, enumerated so that $\widehat T_1=\operatorname{conv}(x_0,x_1,x_2)$, $\widehat T_2=\operatorname{conv}(x_0,x_2,x_3)$, $\dots$, $\widehat T_{n_e}=\operatorname{conv}(x_0,x_{n_e},x_1)$. Consider a reference simplex $S:=\{(s,t)\in\mathbb R^2: s\in[0,1], t\in[0,s]\}$ and observe that each element $\widehat T_j$
can be obtained from $S$ through the affine transformation $(s,t)\mapsto x_0 + s(x_j-x_0)+t(x_{j+1}-x_j)$ whose Jacobian is constant (we are using the logical convention $x_{n_e+1}=x_1$). The shape functions $\widehat\varphi_0$, $\widehat\varphi_j$ and $\widehat\varphi_{j+1}$, restricted to the element $\widehat T_j$, are such that:
\begin{alignat}{2}
\widehat\varphi_0\big(x_0 + s(x_j-x_0)+t(x_{j+1}-x_j)\big) & = 1-s,\notag\\
\widehat \varphi_j\big(x_0 + s(x_j-x_0)+t(x_{j+1}-x_j)\big)  & = s-t,\notag\\
\widehat \varphi_{j+1}\big(x_0 + s(x_j-x_0)+t(x_{j+1}-x_j)\big) & =t,\notag
\end{alignat}
for all $(s,t)\in S$. For a given constant $\kappa\neq 0$ (to be determined later), we propose the following construction of $\widehat\psi_i$ restricted to the element $\widehat T_j$:
$$
\widehat\psi_i\Big|_{\widehat T_j}\big(x_0 + s(x_j-x_0)+t(x_{j+1}-x_j)\big) = \kappa (s-1)(5s-3), \quad\forall (s,t)\in E.
$$
Observe that $\widehat\psi_i$ is invariant with respect to the parameter $t$, and that $\widehat\psi_i$ vanishes when restricted to the exterior edge $\operatorname{conv}(x_j,x_{j+1})$, i.e., when $s=1$. Moreover, this construction is continuous across interior edges 
$\operatorname{conv}(x_0,x_j)={\widehat T_j}\cap{\widehat T_{j-1}}$. Indeed,
\begin{alignat}{2}
\widehat\psi_i\Big|_{\widehat T_j}\big(x_0 + s(x_j-x_0)+0(x_{j+1}-x_j)\big) = \kappa (s-1)(5s-3),\tag{since $t=0$}\\
\widehat\psi_i\Big|_{\widehat T_{j-1}}\big(x_0 + s(x_{j-1}-x_0)+s(x_{j}-x_{j-1})\big) = \kappa (s-1)(5s-3).\tag{since $t=s$}
\end{alignat}
Furthermore, $\widehat\psi_i$ is element-wise orthogonal to the shape functions 
$\{\widehat\varphi_1,...,\widehat\varphi_{n_e}\}$. Indeed,
\begin{alignat}{2}
\int_{\widehat T_j}\widehat\varphi_j\widehat \psi_i= & \big|\widehat T_j\big|{\kappa\over2} \int_0^1\int_0^s(s-t)(s-1)(5s-3)dt\,ds= \big|\widehat T_j\big|{\kappa\over4} \int_0^1{s^2}(s-1)(5s-3)ds,\notag\\
\int_{\widehat T_j}\widehat\varphi_{j+1}\widehat \psi_i= & \big|\widehat T_j\big|{\kappa\over2} \int_0^1\int_0^st(s-1)(5s-3)dt\,ds= \big|\widehat T_j\big|{\kappa\over4} \int_0^1{s^2}(s-1)(5s-3)ds,\notag
\end{alignat}
where it is easy to see that the integral on the right-hand-side vanishes. Finally, we observe that the integral $\int_{\widehat T_j}\widehat\varphi_0\widehat\psi_i$ do not vanish. So we can adjust the constant $\kappa\neq 0$ to get $\int_{\widehat P_i}\widehat\varphi_0\widehat\psi_i=1$, as desired.

Now, for each element $T\subset P_i$ define $\psi_i\big|_T=\widehat\psi_i\circ F_T^{-1}$ (see Assumption~\ref{ass:mesh}) and take a nodal trial function $\varphi_j$. If the support of $\varphi_j$ does not intersect $P_i$, then 
$\int_\Omega\varphi_j\psi_i=0$. Otherwise, there must be $j^*\in\{0,1,...,n_e\}$ such that $\varphi_j\big|_T=\widehat\varphi_{j^*}\circ F_T^{-1}$, for all $T\subset P_i$. In that case, using the change-of-variables theorem and equation~\eqref{eq:scaling}, we have:
$$
\int_\Omega\varphi_j\psi_i=\sum_{T\subset P_i}\int_T\varphi_j\psi_i=\sum_{\widehat T\subset \widehat P_i}\int_{\tilde T}
\widehat\varphi_{j^*}\widehat\psi_i {|T|\over|\widehat T|}=\eta_i\int_{\widehat P_i}\widehat\varphi_{j^*}\widehat\psi_i=\eta_i\delta_{0j^*}=\eta_i\delta_{ij},
$$
since the case $j^*=0$ occurs exactly when $j=i$.

To estimate the norm of $\psi_i$,  first observe that for each element $T\subset P_i$:
$$
\|\nabla\psi_i\|_{L^q(T)}\lesssim \|A_T^{-1}\| |\det A_T|^{1\over q} \|\nabla\widetilde\psi_i\|_{L^q(\widehat T)}
\lesssim {\eta_i^{1\over q}\over \displaystyle\min_{T\in P_i}\rho_T}\|\nabla\widetilde\psi_i\|_{L^q(\widehat T)}.
$$
Hence,
$$
\|\nabla\psi_i\|_{L^q(P_i)}^q=\sum_{T\subset P_i}\|\nabla\psi_i\|_{L^q(T)}^q
\lesssim {\eta_i\over \displaystyle\min_{T\in P_i}\rho_T^q}\sum_{\widehat T\subset\widehat P_i}\|\nabla\widetilde\psi_i\|_{L^q(\widehat T)}^q=
{\eta_i\over \displaystyle\min_{T\in P_i}\rho_T^q}\|\nabla\widetilde\psi_i\|_{L^q(\widehat P_i)}^q\,,$$
which leads to estimate~\eqref{eq:psi_estimate}.
\end{proof}
Now, let us prove Proposition~\ref{prop:compatible2}. 
Recall that we are under the hypothesis that $\mathcal T_h=\{T_i\}_{i=1}^n$ corresponds to a family of shape-regular simplicial meshes, and $q>d$.
Let $\Pi_1:V\to \mathbb P_1(\mathcal T_h)$ be the \emph{Lagrange interpolant} operator, satisfying (for any $T_i\in\mathcal T_h$) the local estimation~\eqref{local::intepor}, together with the global boundedness property~\eqref{bound::interpolant}.
Consider the (Fortin) operator $\Pi: V\mapsto V_m$ defined by:
\begin{equation}\label{eq:Pi2}\notag
\Pi v:=\Pi_1 v + \sum_{j=1}^{N_v}\alpha_j \psi_j,
\end{equation}
where $\psi_j\in\mathbb P_2(\mathcal T_h)$ satisfies~\eqref{eq:orthogonality} and $\alpha_j\in\mathbb R$ is chosen such that the following equality holds true:
\begin{equation}\label{Fortin::eq11}
\int_{\Omega} \varphi_i\,\Pi v=\int_\Omega \varphi_i\,v,\qquad \forall v\in V,\quad\forall i=1,...,N_v.
\end{equation}
Solving for $\alpha_i$ we get
\begin{equation}\label{eq:alpha_i}
\alpha_i={1\over \eta_i}\int_\Omega\varphi_i(v-\Pi_1 v).
\end{equation}
Observe that~\eqref{eq:Pi2} implies property~\eqref{Fortin:c3} of the Fortin operator. To prove property~\eqref{constC} we start estimating $\alpha_i$. So let $P_i:=\overline{\text{supp}\,\varphi_i}$ be the patch of simplices supporting the nodal basis function $\varphi_i$. We have:
$$
|\alpha_i|\leq {1\over \eta_i}\sum_{T\in P_i} \|\varphi_i\|_{L^p(T)}\|v-\Pi_1v\|_{L^q(T)}\lesssim
\sum_{T\in P_i}{ |T|^{1\over p}\over\eta_i}h_T \|\nabla v\|_{L^q(T)}\lesssim {h_i\over\eta_i}|P_i|^{1\over p}\|\nabla v\|_{L^q(P_i)},
$$
where $h_i=\max_{T\in P_i}h_T$. 
Let $\tilde T\subset P_i$ such that $\rho_{\tilde T}=\min_{T\in P_i}\rho_T$. By quasi-uniformity of the patches, there is a mesh-independent constant $c>0$, such that $h_i\leq c h_{\tilde T}$. Hence, by shape-regularity we have that $h_i/\rho_{\tilde T}\lesssim h_{\tilde T}/\rho_{\tilde T}$ is uniformly bounded from above.
Next, using the estimate~\eqref{eq:psi_estimate} we get:
$$
|\alpha_i|\|\nabla \psi_i \|_{L^{q}(P_i)}\lesssim {|P_i|^{1\over p}\over\eta_i^{1\over p}}{h_i\over\rho_{\tilde T}}\|\nabla v\|_{L^q(P_i)}\lesssim \left(\sum_{T\subset P_i}{|T|\over\eta_i}\right)^{1\over p}
\|\nabla v\|_{L^q(P_i)}=|\widehat P_i|^{1\over p} \|\nabla v\|_{L^q(P_i)}.
$$
This leads to the estimate:
\begin{equation}
\sum_{i=1}^{N_v}|\alpha_i|^q\|\nabla \psi_i \|^q_{L^{q}(P_i)}%
\lesssim\sum_{i=1}^{N_v} \|\nabla v\|^q_{L^q(P_i)}\leq (d+1)\|\nabla v\|^q_{L^q(\Omega)}.
\end{equation}
Finally, we have:
\begin{align*}
\|\nabla  \Pi v\|^q_{L^q(\Omega)} &= \sum_{T\in\mathcal T_h}\Big\|\nabla\Pi_1v+\sum_{\{i/T\subset P_i\}}
\alpha_i\nabla\psi_i\Big\|^q_{L^q(T)}\\
& \leq (d+2)^{q-1}\sum_{T\in\mathcal T_h} \left(  \|\nabla \Pi_1 v \|^q_{L^q(T)} +\sum_{\{i/T\subset P_i\}}|\alpha_i|^q \| \nabla \psi_i \|^q_{L^q(T)}\right)\\%
&\lesssim  \| \nabla \Pi_1 v\|^q_{L^q(\Omega)}+ \sum_{i=1}^{N_v}|\alpha_i|^q\|\nabla \psi_i \|^q_{L^{q}(P_i)}\lesssim  \|\nabla v\|^q_{L^q(\Omega)}.
\end{align*}

\section*{Acknowledgements}
\addcontentsline{toc}{section}{Acknowledgements}
The authors want to thank Diego Paredes for helping with preliminary numerical experiments. 
The work by IM and FM was done in the framework of Chilean FONDECYT research project~\#1160774. IM and SR have also received funding from the European Union's Horizon 2020 research and innovation programme under the Marie Sklodowska-Curie grant agreement No 777778 (MATHROCKS). 
The research by KvdZ was supported by the Engineering and Physical Sciences Research Council (EPSRC) under grant EP/T005157/1. 

\bibliographystyle{spmpsci}
\bibliography{mybib}

\end{document}